\documentclass[journal]{IEEEtran}
\usepackage{amssymb,epsfig,xcolor,cite,amsmath,amsfonts,mathrsfs,algorithm}
\usepackage{epstopdf}
\usepackage{datetime,fancyhdr}
\usepackage{algpseudocode}
\usepackage{arydshln}
\usepackage{multirow}
\usepackage{amsthm} 
\allowdisplaybreaks
\usepackage{times} % assumes new font selection scheme installed
\usepackage{epsfig}
\usepackage{graphicx}
\usepackage{latexsym}
\usepackage{subfigure}
\usepackage{threeparttable}

\newtheorem{lemma}{Lemma}[section]
\newtheorem{theorem}{Theorem}[section]
\newtheorem{corollary}{Corollary}[section]
\newtheorem{remark}{Remark}[section]
\newtheorem{definition}{Definition}[section]

\newtheorem{proposition}{Proposition}[section]
\newtheorem{assumption}{Assumption}[section]
\definecolor{purple}{RGB}{128,0,128} 
\ifCLASSINFOpdf
\else
\fi

\begin{document}

\title{Predictive Online Disturbance-Action Control for Linear Dynamical Systems}

\author{Xia Jiang, Jianping Li, Lihua Xie,~\IEEEmembership{Fellow,~IEEE}
% <-this % stops a space
\thanks{This paper was not presented at any conference. Corresponding author is Lihua Xie.}
\thanks{X. Jiang (xia.jiang@ntu.edu.sg) and J. Li (jianping.li@ntu.edu.sg) are with the School of Electrical and Electronic Engineering, Nanyang Technological University, Singapore 639798.}
\thanks{L. Xie (elhxie@ntu.edu.sg) is with NTU–VinUni Joint Research Laboratory for Embodied AI and Robotics, School of Electrical and Electronic Engineering, Nanyang Technological University, Singapore 639798 and VinUniversity, Hanoi, Vietnam.}
}

\maketitle

{
\begin{abstract}
    This paper studies predictive online control for linear dynamical systems with time-varying cost functions, motivated by applications such as real-time adaptive control of motorized LiDAR sensing systems, where the controller must adjust the sensing direction online to balance localization accuracy, scanning efficiency, and smooth actuation under changing environmental conditions. We propose a predictive online control (POC) algorithm that leverages short-term predictions of future cost functions while accounting for the memory effect induced by system dynamics. Using the disturbance-action controller (DAC) parameterization, the online control problem is transformed into an OCO-with-memory formulation over policy parameters. We develop a windowed receding-horizon update that incorporates short-term predictions and accommodates a total-variation regularizer to suppress abrupt policy variations. Theoretically, we prove that POC achieves a dynamic policy regret bound scaling with the path length of the comparator sequence, and provide sufficient conditions in terms of logarithmic prediction and memory horizons for the regret guarantee. The proposed method is evaluated on a motorized LiDAR sensing task, demonstrating improved localization accuracy together with a favorable trade-off between sensing accuracy and scanning completeness.
\end{abstract}}
% \begin{abstract}
%  This paper studies the online control of linear dynamical systems with adversarial cost functions, which is a fundamental problem in online learning and control theory. We propose a novel predictive online control (POC) algorithm that leverages short-term predictions of future cost functions to optimize the dynamic policy regret of online control. The online control problem is transformed into online optimization with memory under the disturbance-action controller (DAC) parameterization, which captures the memory effect of the control system. By incorporating a carefully designed receding horizon objective that balances stage costs with switching costs that penalize abrupt changes, the proposed algorithm utilizes short-term predictions to obtain superior regret performance. Our analysis proves that the proposed POC algorithm achieves a dynamic policy regret bound that scales with the path-length of the comparator sequence, providing theoretical guarantees for its performance in non-stationary environments. We further validate our theoretical results through simulations on motorized LiDAR sensing, showcasing the practical effectiveness of our approach in real-world applications.
 
%  \end{abstract}

\begin{IEEEkeywords}
online optimization with memory, predictive control, disturbance-action controller, dynamic policy regret
\end{IEEEkeywords}

% \fi
\IEEEpeerreviewmaketitle

\section{Introduction}

%online optimization/online control  and its applications

\par Modern robotic applications, such as path planning \cite{path_planning_us} and perception navigation \cite{PAMPC_perce_awa}, inherently operate in dynamic, highly non-stationary environments. This motivates the growing interest in online control, which integrates online learning with control theory to enable real-time adaptation. Conventional control paradigms have notable limitations in such settings. Traditional optimal control methods (e.g., LQG \cite{Bertsekas2012DynamicProgramming}) typically assume known, static cost functions and specific noise statistics, whereas robust control methods (e.g. $H_\infty$ \cite{Lee2020AdaptiveRobust}) often produces overly conservative policies by optimizing against worst-case scenarios. In constrast, online control \cite{Hazan2022OnlineControl} provides a more flexible framework that allows for time-varying, sequentially revealed, and potentially adversarial cost functions, where the controller must adapt in real-time without prior knowledge of future objectives. Moreover, online control shifts the objective from minimizing an absolute expected cost to minimizing the performance gap with respect to the best fixed policy in hindsight. This regret-based formulation leverages online convex optimization (OCO) to handle exogenous inputs and evolving constraints \cite{Hazan2022OnlineControl}. As a result, it is particularly suitable for real-world systems, where cost functions vary with environmental conditions. %The online control framework offers a theoretically rigorous foundation that complements modern adaptive control and reinforcement learning algorithms, enabling them to effectively operate in dynamic and uncertain environments.

%online optimization with memory, policy regret
\par Online control of linear dynamical systems has been extensively studied in the literature \cite{Agarwal2018OnlineControl,agarwal2019logarithmic}, 
yielding various algorithms with regret guarantees under different settings. Because cost functions are revealed online and may be adversarial, the optimal policy is generally unavailable a priori, and performance is therefore commonly evaluated via policy regret against the best fixed policy in hindsight. 
The seminal work \cite{Agarwal2018OnlineControl} established an $\mathcal{O}(\sqrt{T})$ policy regret bound for linear dynamical systems by reducing the online control problem to online convex optimization with memory. Subsequent studies refined these guarantees to achieve logarithmic regret for strongly convex losses \cite{Simchowitz2020, agarwal2019logarithmic, Foster2020} and extended the framework to accommodate action constraints \cite{Li2021, Liu2023}. Although these methods provide tight performance boundaries under strongly convex objectives, they are confined to static regret metrics and fail to exploit lookahead predictions. Related extensions have also considered more general settings, such as unknown system dynamics \cite{Hazan2020}, bandit feedback \cite{Gradu2020}, and competitive control ratios \cite{Shi2020, Goel2023CompetitiveControl}. %Most of these results focus on fully observed systems, while \cite{Simchowitz2020Improper} provides insight into partially observed settings. There are also some other related works studying online control from the perspectives of competitive ratio \cite{Shi2020,Goel2023CompetitiveControl} and adaptive regret \cite{Hazan2009EfficientLearning,Daniely2015}.

%reduction from online non-stochastic control to online convex optimization with memory, which allows for the application of powerful online learning techniques to control problems. Specifically, the reduction leverages a class of policies called disturbance-action controllers (DACs) to represent control actions as linear functions of past disturbances. This formulation captures the inherent memory in control problems, where the current state and action depend on a sequence of past decisions and disturbances. In addition, the framework reduces the unbounded memory problem to a tractable online optimization formulation with fixed memory windows.  

%By defining a truncated loss function that considers only a fixed window of past policies, the reduction effectively transforms the online control problem into an online optimization problem with memory []. This allows for the design of algorithms that can achieve low policy regret by optimizing over the truncated loss, while ensuring that the approximation error introduced by truncation is controlled. The reduction to OCO with memory provides a powerful framework for analyzing and designing online control algorithms that can adapt to non-stationary environments. 

%linear dynamical systems works, onine nonstochastic control (但是本文是stochastic noise设定)
\par Most existing studies focus on static policy regret, which only measures performance against a fixed comparator and is not suitable for learning in non-stationary and open environments \cite{Sugiyama2012}. In contrast, dynamic policy regret benchmarks the algorithm against a sequence of time-varying policies and better captures the inherent difficulty of online control, where the optimal policy evolves in response to changing cost functions. Pioneering work \cite{Zinkevich2003}  established dynamic regret analysis for online convex optimization, showing that online gradient descent attains an $\mathcal{O}(\sqrt{T}(1+P_T))$ dynamic regret, where $P_T$ denotes the path-length of the comparator sequence. More recently, the work \cite{Zhao2023NonStationaryOnlineLearning} developed a meta-base decomposition algorithm for non-stationary online learning with memory and applied it to online control, achieving an optimal dynamic policy regret bound. Despite these advances, the above algorithms do not explicitly exploit lookahead predictions of future cost functions, which may be useful for improving performance in non-stationary environments.

%online optimization with prediction (optimistic and short-term prediction, MPC)
\par To incorporate prediction capabilities into online control, we consider online optimization with prediction, where the decision maker receives short-term forecasts of future cost functions. This setting is motivated by the observation that in many real-world applications, near-term predictions are often more reliable than long-term forecasts. Recent application-oriented predictive control studies have shown strong practical potential in real-time engineering systems, including data-driven predictive control for microgrid energy management \cite{data_predic_vol}. These works demonstrate the
benefits of combining predictions with online adaptation, motivating
the need for predictive online control methods with explicit
performance guarantees. In the
online optimization literature, two representative approaches have been
studied, including optimistic online learning \cite{pmlr-v139-flaspohler21a,GrettonAccessible2016,pmlr-v30-Rakhlin13} and receding horizon control \cite{data_driven_motion,Li2021PredictOptimization,Rawlings2009}. Optimistic online learning algorithms exploit gradient predictions to improve regret bounds, while receding horizon control, including model predictive control (MPC), computes control actions by optimizing over a finite lookahead window of predicted cost functions. However, MPC approaches generally require to solve multi-stage optimization problems at each time step, which can be computationally prohibitive for real-time control applications \cite{Li2021PredictOptimization}. Despite the efforts to reduce computational burden \cite{preriod_event_mpc,Graichen2010StabilityIncremental,Alessio2009,Zeilinger2011,Paternain2018,Diehl2005}, these MPC algorithms lack regret guarantees for online problems with time-varying costs and do not directly apply to online control, where cost functions are arbitrary and the current cost depends on past decisions through state transitions.
%%online control with prediction
{
\par More recently, the incorporation of predictive information has shown considerable potential to improve sequential decision-making. For example, the work \cite{yan2024incorporation} effectively incorporates agents' likely future actions into pseudo-gradient dynamics for noncooperative games, enabling proactive decision-making and showing that predictive information can help characterize equilibrium stability under suitable conditions. In the online control setting, \cite{Mhaisen2024OptimisticOnline} incorporates gradient predictions into an optimistic Follow-the-Regularized-Leader (FTRL) framework, while prediction-aided and structured online control methods \cite{NEURIPS2020_155fa095,NEURIPS2019_6d1e481b,zhang2021regretanalysisonlinelqr,learning_regret_onlinequa} exploit finite lookahead windows or specific structural properties to improve sequential decision-making. However, predictive online control with non-stationary costs and time-varying policy benchmarks remains challenging.
The optimistic online control framework in \cite{Mhaisen2024OptimisticOnline} establishes guarantees only for static policy regret and does not explicitly account for policy variations. 
Existing methods \cite{NEURIPS2020_155fa095,NEURIPS2019_6d1e481b,zhang2021regretanalysisonlinelqr,learning_regret_onlinequa} typically rely on smooth finite-horizon objectives, smooth switching penalties, or specific LQR/quadratic formulations. They do not directly address the DAC-based OCO-with-memory formulation considered here, where the current loss depends on a truncated history of policy parameters induced by past disturbances and state transitions.}
%In our general setting, the predictive objective is designed specifically for dynamic policy regret minimization while explicitly regulating policy variation. This leads to a non-smooth predictive optimization problem that cannot be directly handled by existing smooth predictive frameworks, thereby requiring a different algorithmic treatment and theoretical analysis.
% since dynamic policy regret minimization requires explicit regulation of policy variation, resulting in a non-smooth optimization problem that calls for different algorithmic treatment and theoretical analysis.
% Second, although recent prediction-aided methods have been applied to linear dynamical systems, they rely on smooth finite-horizon objectives. Extending these smooth frameworks to general non-stationary online control is nontrivial. To achieve strict dynamic policy regret guarantees, it is crucial to explicitly penalize abrupt policy changes. This naturally introduces a non-smooth total-variation regularization, which, when coupled with system dynamics, renders standard smooth predictive methods inadequate and makes our objective formulation and algorithm design fundamentally different.

\par 
The main contributions of this work are summarized as follows.
\begin{itemize}
    \item We propose a predictive online control (POC) algorithm for linear dynamical systems with short-term cost predictions, motivated by applications such as real-time adaptive control of motorized LiDAR sensing systems.
    Using the DAC parameterization, we reformulate the control problem as OCO with memory, capturing the history-dependent effect of past disturbances and state transitions on current costs. The designed windowed receding-horizon update combines an online gradient descent (OGD) warm start with joint constrained proximal-gradient refinement, exploiting lookahead information while accommodating the non-smooth total-variation regularizer on successive DAC policy parameters. This regularizer promotes temporally consistent policy updates and smoother closed-loop behavior.
	% \item This paper proposes a novel Predictive Online Control (POC) algorithm that integrates short-term predictions into non-stationary online control. While seminal works (e.g., [27], [34]) have explored predictions in standard Online Convex Optimization (OCO), extending this to dynamical systems is highly non-trivial due to the unbounded memory effect of state transitions. By leveraging the Disturbance-Action Controller (DAC) parameterization, we rigorously reduce the online control problem to an OCO with memory and introduce a receding-horizon objective to exploit predictions effectively.
	
    % \item We develop a predictive online control (POC) algorithm that combines an OGD warm start with joint constrained proximal-gradient refinement over a finite lookahead window. Unlike receding-horizon gradient methods for smoothed OCO \cite{Li2021PredictOptimization,li_pre_nips_smooth}, which typically rely on smooth squared switching penalties, POC handles the non-smooth total-variation penalty. This penalty directly suppresses abrupt changes in the DAC parameters and better reflects the switching behavior of online control policies, but it also introduces coupled non-smooth window subproblems that require a proximal treatment.
    \item We establish a dynamic policy regret guarantee for the proposed POC algorithm by connecting the finite-window receding-horizon optimization with a global regularized objective. Our analysis explicitly characterizes the coupled effects of prediction horizon and DAC memory truncation on the regret performance. Under standard assumptions, choosing both the memory length $H$ and the prediction horizon $W$ on the order of $\mathcal{O}(\log T)$ effectively controls the memory-truncation and finite-window optimization errors, leading to a path-length-dependent dynamic policy regret guarantee.  Compared with existing predictive online control methods \cite{Mhaisen2024OptimisticOnline,NEURIPS2020_155fa095,NEURIPS2019_6d1e481b,zhang2021regretanalysisonlinelqr}, our analysis provides a regret characterization that simultaneously accounts for lookahead information and the memory-dependent policy dynamics.

	%\item We provide a rigorous theoretical analysis showing that the proposed POC algorithm achieves a dynamic policy regret bound of $\mathcal{O}(P_T)$ with a prediction horizon of $\mathcal{O}(\log T)$, matching the optimal regret order for strongly convex cost functions. This improves upon existing approaches, such as optimistic FTRL \cite{Mhaisen2024OptimisticOnline}, that only guarantee static regret. Our analysis integrates the convergence properties of constrained proximal gradient methods with classical results from online gradient descent, offering new  insights into exploiting predictive information for performance improvement.
    {
	\item We validate the proposed POC algorithm on a motorized LiDAR sensing (MLiS) task using simulation environments based on the MCD dataset. The empirical results demonstrate that POC improves localization accuracy while achieving a favorable accuracy-scanning trade-off and lower computation time than MPC-based baselines. The total-variation regularization is designed to promote smoother motor-speed transitions, supporting the suitability of POC for real-time sensing-control applications.} 
% the practical effectiveness of the proposed POC algorithm through extensive simulations on a motorized LiDAR sensing (MLiS) system. The empirical results show that POC achieves superior localization accuracy and scanning efficiency compared to baseline methods, while its total-variation regularization yields smoother control updates. These results demonstrate the practical benefits of incorporating short-term predictions into dynamic online control.
	%\item We demonstrate the practical effectiveness of the proposed method through extensive simulations on a motorized LiDAR sensing (MLiS) system. Unlike conventional online control approaches that ignore predictive information and focus mainly on static regret, POC explicitly leverages short-term predictions to optimize dynamic policy regret while characterizing their impact, making it well-suited for non-stationary environments with evolving cost functions and disturbances. The empirical results show that POC achieves superior localization accuracy and scanning efficiency compared to baseline methods, validating the practical benefits of our theoretical framework in real-world control applications.
\end{itemize}

The remainder of this paper is organized as follows. Section \ref{solver_design} describes the online control problem formulation and the design of the predictive online control algorithm, which also intoduces some preliminary concepts and the reduction from online control to OCO with memory. Section \ref{analysis_sec} provides the policy regret analysis for the proposed algorithm. Section \ref{sim_sec} demonstrates the performance of the proposed algorithm through motorized LiDAR sensing simulations. Finally, Section \ref{conclu_sec} concludes the paper and discusses future research directions. 

\subsection{Mathematical notations}
We denote $\mathbb{R}$ as the set of real numbers, $\mathbb{R}^n$ as the set of $n$-dimensional real column vectors, $\mathbb{R}^{n\times m}$ as the set of $n$-by-$m$ real matrices. For the control system, we reserve the letters $x,y$ for states and $u,v$ for actions. We denote $d_x$ and $d_u$ as the dimension of state and control action, respectively. Let $d=\max\{d_x,d_u\}$. The notation $M_{i:j}$ denote a sequence of variables $M_i, M_{i+1}, \cdots, M_j$. 
 The notation $\langle\cdot,\cdot \rangle$ denotes the inner product.   For a differentiable function $f(x)$, $\nabla f(x)$ denotes the gradient of $f$ with respect to $x$. 
\section{Problem description and algorithm design}\label{solver_design}
\par We study the online control framework with a discrete-time dynamical system, where at each time $t$, the controller observes the system state $x_t\in \mathbb{R}^{d_x}$ and decides an action $u_t \in \mathbb{R}^{d_u}$, which then induces a stage cost $c_t(x_t,u_t)$, and causes a transition to a new state $x_{t+1}$. In the online setting, the cost and the new state are revealed to the controller after it commits its action. In this paper, we study the online control of the linear dynamical system governed by
\begin{align}\label{linear_dyn_sys}
x_{t+1}=Ax_t+Bu_t+w_t,
\end{align}
where $A\in \mathbb{R}^{d_x\times d_x}$, $B\in \mathbb{R}^{d_x\times d_u}$ and the disturbance vector $w_t\in \mathbb{R}^{d_x}$. %In online nonstochastic control [], the disturbance can be generated arbitrarily and no statistical assumption is imposed on its distribution; additionally, cost functions can be chosen adversarially. The adversarial disturbance and online cost functions hinder the priori computation of the optimal policy in classical control theory [] and requires online learning techniques to handle adversarial environments.
% \par To define the performance metric, we denote the cost for a control algorithm $\mathcal{A}$ as 
% $$J_T(\mathcal{A})=\sum_{t=0}^T c_t(x_t,u_t).$$ 
% \par The standard measure for online control is the \textit{static policy regret}, defined as 
% \begin{align*}
% R_T=\sum_{t=1}^T c_t(x_t,u_t)-\min_{\pi\in \Pi}\sum_{t=1}^T c_t(x_t^{\pi},u_t^{\pi}).
% \end{align*}
% In static policy regret, the comparator $\pi$ could be chosen with complete knowledge of the disturbance and loss functions. However, since the unknown disturances and cost functions can change arbitrarily, the optimal controller of each round would also change accordingly. 
% Hence, it is necessary to adopt the \textit{dynamic policy regret} to benchmark the algorithm with a sequence of time-varying controllers $\pi_1,\cdots,\pi_T \in \Pi$, i.e.,
% \begin{align}
% DR_T=\sum_{t=1}^T c_t(x_t,u_t)-\sum_{t=1}^T c_t(x_t^{\pi_t},u_t^{\pi_t}).
% \end{align}
{
\par The standard measure for online control is evaluated via the expected dynamic policy regret, which benchmarks the algorithm against a sequence of time-varying controllers $\pi_1, \dots, \pi_T \in \Pi$, i.e.,
\begin{equation}
\label{eq:dr}
DR_T = \mathbb{E}\left[ \sum_{t=1}^T c_t(x_t, u_t) - \sum_{t=1}^T c_t(x_t^{\pi_t}, u_t^{\pi_t}) \right],
\end{equation}
where the expectation is taken with respect to the stochastic disturbances.
}
In addition, the benchmark set $\Pi$ is chosen as the class of disturbance-action controllers, which encompasses many controllers of interest.
%{\color{blue}The regret has been used in the literature [].}
{
\begin{remark}
Our dynamic regret is benchmarked against the best sequence of DAC policies. This is theoretically equivalent to comparing against the broad class of strongly stable linear policies, since any such controller can be approximated arbitrarily well by a DAC policy with a proper memory length \cite{Agarwal2018OnlineControl}. %This approximation gap is formally quantified as the truncation error in Lemma \ref{trunc_approx_err}.
Furthermore, unlike memoryless online optimization with instantaneous action minimizers, our policy-based benchmark respects the closed-loop state transitions and captures the effect of
current actions on future states.
\end{remark}
}
\subsection{Reduction to OCO with memory}\label{red_oco_sec}
This section describes the evolution of \eqref{linear_dyn_sys} under a non-stationary pollicy sequence of length $T$. Following the pioneering work \cite{Agarwal2018OnlineControl}, we adopt the disturbance-action controller (DAC) policy class, which parametrizes the executed action as a linear function of past disturbances.  This formulation enables a reduction of the online control problem to OCO with memory.

\begin{definition}
	(Disturbance-Action Controller, DAC). A disturbance-action controller $\pi(K, M)$ with a memory length $H>1$ is specified by a fixed matrix $K$ and parameters $M=\left(M^{[0]}, \ldots, M^{[H-1]}\right)$. At each time $t$, the policy $\pi(K, M)$ chooses the action $u_t$ at a state $x_t$, defined as $u_t=-K x_t+\sum_{i=1}^H M^{[i-1]} w_{t-i}$.
\end{definition}
\par For convenience, we define $w_i=0$ for $i<0$.  Since the system matrices $A$ and $B$ are assumed to be known, the disturbance can be exactly recovered as $w_t=x_{t+1}-A x_t-B u_t$, which makes the DAC policy implementable. We denote the policy applied at time $t$ by $M_t=\{M_t^{[i]}\}$, where the subscript $t$ indicates the time index and the superscript $[i-1]$ denotes the coefficient acting on $w_{t-i}$. The DAC policy generalizes the linear controller $u_t=-K x_t$ and is rich enough to represent a broad class of controllers \cite{Hazan2022OnlineControl}, including the optimal linear policy in the LQR setting \cite{Agarwal2018OnlineControl}. The following proposition states an important property of DAC policies.

\begin{proposition}
	Suppose the DAC controller $\pi\left(K, M_t\right)$ is applied at time $t$. Then the reached state satisfies the following recurrence for any $h\geq 0$, 
	$$x_{t+1}=\widetilde{A}_K^{h+1} x_{t-h}+\sum_{i=0}^{H+h} \Psi_{t, i}^{K, h}\left(M_{t-h: t}\right) w_{t-i},$$ 
	%and $$u_t^K\left(M_{0: t}\right)=-K x_t^K\left(M_{0: t-1}\right)+\sum_{i=1}^H M_t^{[i]} w_{t-i},$$ 
	where $\widetilde{A}_K=A-B K$ and $$\Psi_{t, i}^{K, h}\left(M_{t-h: t}\right)=\widetilde{A}_K^i \mathbf{1}_{i \leq h}+ \sum_{j=0}^h \widetilde{A}_K^j B M_{t-j}^{[i-j-1]} 1_{1 \leq i-j \leq H}.$$
\end{proposition}
The matrix $\Psi_{t, i}^{K, h}\left(M_{t-h: t}\right)$ is a transfer matrix that describes the effect of $w_{t-i}$ with respect to the past $h+1$ policies on the state $x_{t+1}$. When $M$ is the same across all arguments we compress the notation to $\Psi_{t, i}^{K, h}(M)$.
\par In the control setting, the loss incurred at time step $t$ depends on all past decisions due to the counter-factual nature of regret. Since both the state $x_t$ and control input $u_t$ are linear functions of the DAC parameters $M_0, \ldots, M_t$, the cost $c_t\left(x_t^K\left(M_{0: t-1}\right), u_t^K\left(M_{0: t}\right)\right)$ becomes a linear function of the historical parameters $M_{0: t}$. This introduces a fundamental challenge that the effective memory length grows unbounded over time, which is infeasible in the OCO with memory framework. To address this, our optimization approach considers only the effects of the past $H$ steps during planning, disregarding the system state at time $t-H$. We will show later that this truncation scheme tracks the true cumulative cost with only a small approximation error. To formalize this idea, we introduce the notion of a truncated state and truncated loss, which enable a finite-memory approximation of the original dynamical system. 
{
	\begin{definition}\label{f_t_def}
		(Truncated Loss). For the cost function $c_t: \mathbb{R}^{d_x} \times \mathbb{R}^{d_u} \mapsto \mathbb{R}$ and DAC policies $\left\{\pi\left(K, M_t\right)\right\}_{t=1, \ldots, T}$, given a memory length $H$, the induced truncated loss $f_t: \mathcal{M}^{H+2} \mapsto \mathbb{R}$ is defined as follows.
		\par For $t>H$,
		$$
		f_t\left(M_{t-1-H: t}\right)=\mathbb{E} \left[ c_t\left( y_t^K(M_{t-1-H:t-1}), v_t^K(M_{t-1-H:t}) \right) \right],
		$$
		where the expectation is taken over $\{w_{t-i}\}_{i=0}^{2H}$, the truncated state and truncated DAC control are $$y_{t+1}^K(M_{t-H:t})=\sum_{i=0}^{2 H} \Psi_{t, i}^{K, H}\left(M_{t-H: t}\right) w_{t-i},$$ and 
		$$
		v_{t+1}^K(M_{t-H:t+1})=-K y_{t+1}^K\left(M_{t-H: t}\right)+\sum_{i=1}^H M_{t+1}^{[i-1]} w_{t+1-i} .
		$$
		\par For the initial stages $t\leq H$, the truncated loss is defined analogously by using all available past disturbances and policy parameters:
		$$
		f_t\left(M_{0: t}\right)=\mathbb{E} \left[ c_t\left( y_t^K(M_{0: t-1}), v_t^K(M_{0: t}) \right) \right],
		$$
		where the truncated state and control are defined with summations truncated to the available indices, i.e.,
		$$
		y_{t+1}^K\left(M_{0: t}\right)=\sum_{i=0}^{2 t} \Psi_{t, i}^{K, t}\left(M_{0: t}\right) w_{t-i},
		$$
		and 
		$$\quad v_{t+1}^K\left(M_{0: t+1}\right)=-K y_{t+1}^K\left(M_{0: t}\right)+\sum_{i=1}^t M_{t+1}^{[i-1]} w_{t+1-i}. $$
	\end{definition} 
	\par The dependence of the truncated state on disturbances extends to $2H$ steps due to the recursive interaction between the system dynamics and the disturbance-action controller. While the control input depends explicitly on the most recent $H$ disturbances, the state is recursively generated from past controls, which themselves depend on earlier disturbances. Unrolling this recursion shows that disturbances up to $w_{t-2H}$ can influence $y_{t+1}^K$, effectively doubling the memory length.}
The approximation error introduced by the truncation, the discrepancy between $f_t$ and $c_t$, can be precisely bounded, as established in Lemma \ref{trunc_approx_err}. This allows us to replace the original loss with the truncated loss $f_t$, which avoids the issue of unbounded memory while preserving accuracy. As a result, the online control problem admits a well-defined reformulation as an OCO problem with memory, completing the reduction.
\subsection{Online Optimization Algorithm with Prediction}\label{online_pred_sec}
The above reduction provides a foundation for developing efficient online optimization algorithms with memory, which enables the design of online controllers that are competitive with time-varying comparator policies. In many sequential decision-making problems, accurate short-term predictions of future cost functions are often available, while abrupt and large decision changes are undesirable in practice. Motivated by these observations, we further study online optimization with short-term predictions and introduce the following optimization objective: 
\begin{align}\label{rh_obj}
g_T^{RH}(\mathbf{M})=\sum_{t=1}^{T} \tilde{f}_{t}(M_t)+\sum_{t=1}^{T}\lambda  {\|M_{t}-M_{t-1}\|_F},
\end{align}
where $\mathbf{M}=(M_1,\cdots,M_T)$ and $\lambda\geq 0$ is a penalty parameter.
{
The expected surrogate loss is defined by
\begin{equation}
\tilde{f}_t(M) \triangleq f_t(M,\ldots, M),
\end{equation}
where $f_t(\cdot)$ is the truncated loss introduced in Definition \ref{f_t_def}. The cost $\tilde{f}_t(M_t)$ is obtained by re-parameterizing the original state-input cost $c_t(x_t, u_t)$ under the DAC framework, where both the state $x_t$ and the control input $u_t$ can be expressed as linear functions of the policy parameter $M_t$.} {
For notation convenience, we use $M_t^*$ to denote an arbitrary dynamic comparator sequence used to define the dynamic policy regret. In contrast, we define $M_{t,RH}^*$ as the global minimizer of the regularized total cost $g_T^{RH}$ in \eqref{rh_obj}.
}
% We refer to $g_T^{RH}$ in \eqref{rh_obj} as the Global Regularized Objective. The phrase "windowed receding-horizon objective" is reserved for the local shifted objective $g_{t,W}^{RH}$ defined over the lookahead window $[t,t+W-1]$. We also standardize terminology throughout the manuscript by using "time step" to denote the index $t$ and "inner proximal iteration" to denote the iteration index $k$ used inside the lookahead optimization.}
{
\par The objective function \eqref{rh_obj} consists of the stage cost $\tilde f_t(\cdot)$ together with a total-variation regularizer on successive DAC policy parameters. The regularizer is introduced to control temporal variations of the learned disturbance-action policy. Since the DAC policy maps past disturbances to the current control action, abrupt variations in $M_t$ may induce undesirable fluctuations in the closed-loop response. The regularizer therefore enforces temporal consistency of the learned disturbance-action policy and promotes smoother control behavior.
\par 
This formulation differs from existing prediction-aided online control and receding-horizon gradient methods
\cite{NEURIPS2020_155fa095,NEURIPS2019_6d1e481b,zhang2021regretanalysisonlinelqr},
which typically exploit lookahead information through smooth finite-horizon objectives or LQR-specific structures. In contrast, in our setting, the DAC parameterization
converts the control problem into an OCO-with-memory problem, where the loss at time $t$ depends on a truncated history of policy parameters induced by past disturbances and state transitions. For a time-varying comparator sequence, the dynamic policy regret is naturally characterized by the first-order path length, motivating the
total-variation regularizer in \eqref{rh_obj}. The resulting non-smooth receding-horizon subproblem is therefore handled by the joint proximal-gradient refinement developed below.}
\begin{algorithm*}
	\caption{ Predictive online control (POC) algorithm}\label{POC_alg}
	\begin{algorithmic}[1] 
		\State Initialize: Step size $\eta$, Parameters $\kappa_B, \kappa, \gamma, T, \mathcal{M}, H=\lceil 2 \gamma^{-1}\log T \rceil$, {penalty parameter $\lambda=(H+2)^2 L_f$.}
		%\State Define $H=\gamma^{-1} \log \left(T \kappa^2\right)$.
		%\State Define $\mathcal{M}=\left\{M=\left\{M^{[0]} \ldots M^{[H-1]}\right\}:\left\|M^{[i-1]}\right\| \leq \kappa^3 \kappa_B(1-\gamma)^i\right\}$.
		\State Initialize $M_1(0)=M_0\in \mathcal{M}$ arbitrarily.
        \State Receive initial predictions $\tilde{f}_1, \dots, \tilde{f}_{W-1}$.
        \For{$s = 1, \dots, W-1$}
            \State $M_{s+1}(0) = \Pi_{\mathcal{M}} \left[ M_s(0) - \delta \nabla \tilde{f}_s(M_s(0)) \right]$
        \EndFor
		\For {$t=1,\cdots,T$}
		\State Update $M_{t+W}$ by OGD \eqref{OGD_up}.
		\begin{align}\label{OGD_up}
		M_{t+W}(0)=\Pi_{\mathcal{M}}\left[M_{t+W-1}(0)-\delta \nabla {\tilde{f}}_{t+W-1}(M_{t+W-1}(0))\right]
		\end{align}

		\State Set $\mathbf{M}_{t,W}(0) = (M_t(0), M_{t+1}(0), ..., M_{t+W-1}(0))$  %滚动继承上一轮的值，末尾用 OGD 初始化
		\For {$k = 1$ to $W$ }  % Line 7 改变：不再循环时间步 s，而是循环迭代步数 k
		\State \begin{align}\label{inn_prox_up}
		\mathbf{M}_{t,W}(k)
		=&\arg\min_{\mathbf{Z}\in\mathcal{M}^{W}}\bigg\{
		R_t(\mathbf{Z})+\frac{1}{2\eta}
		\big\|\mathbf{Z}
		-\big(\mathbf{M}_{t,W}(k-1)-\eta\nabla F_t(\mathbf{M}_{t,W}(k-1))\big)\big\|_F^2
		\bigg\}.
		\end{align}
		% \State where $F(\mathbf{M}_W)=\sum_{s=t}^{t+W-1} \tilde{f}_s(M_s)$ and $R_t(\mathbf{M}_W)=\sum_{s=t}^{t+W-1} \lambda\left\|M_s-M_{s-1}\right\|_F+\lambda\left\|M_{t+W}^{(0)}-M_{t+W-1}\right\|_F$.
		\EndFor
		\State Set $M_{t}=[\mathbf{M}_{t,W}(W)]^{(1)}$.
        \State Choose the action: $u_t=-K x_t+\sum_{i=1}^H M_{t}^{[i-1]} w_{t-i}$.
		\State Observe the new state $x_{t+1}$ and record $w_t=x_{t+1}-A x_t-B u_t$.
		%\State Online Gradient Update: $M_{t+1}=\Pi_{\mathcal{M}}\left(M_t-\eta_t \nabla f_t\left(M_t\right)\right)$.
	\EndFor
	\end{algorithmic}
\end{algorithm*}
\subsection{Predictive Online Control Algorithm}
\par Building on the reduction framework in Section \ref{red_oco_sec} and the online optimization with prediction introduced in Section \ref{online_pred_sec}, we propose a predictive online control (POC) algorithm for minimizing dynamic policy regret in online control. The complete procedure is presented in Algorithm \ref{POC_alg}. The proposed algorithm consists of two main components: 
\par (1) DAC-based reduction: We adopt the DAC policy $u_t=\pi(K,M_t)$ to represent the control actions, where $K$ is a fixed time-invariant stabilizing matrix. Under this parameterization, the original control problem can be reformulated as an optimization problem over the policy parameter $M_t$. Based on the truncated system representation, the induced surrogate loss $\tilde{f}_t: \mathcal{M}\to \mathbb{R}$ is $\tilde{f}_t(M_t)=f_t(M_t,\cdots, M_t),$
and the decision set is given by 
\begin{align}\label{M_set_def}
\mathcal{M}=\Big\{&M=\big\{M^{[0]} \ldots M^{[H-1]}\big\}:
\notag\\
&\quad \big\|M^{[i-1]}\big\| \leq \kappa^3 \kappa_B(1-\gamma)^i\Big\}.
\end{align}
Here, $M_t=\left(M_t^{[0]}, \ldots, M_t^{[H-1]}\right)$ denotes the collection of disturbance-action parameters over the memory horizon at time $t$, and $\|\cdot\|$ denotes the spectral norm. The constraint set $\mathcal{M}$ admits a separable structure as a Cartesian product of block-wise spectral-norm balls with explicitly specified radii. Moreover, the projection $\Pi_{\mathcal{M}}$ used in the update step is the Euclidean (Frobenius-norm) projection onto $\mathcal{M}$, which decomposes into independent block-wise projections. Each block's projection onto a spectral-norm ball can be computed by SVD-based singular-value clipping, making the projection step computationally tractable.  %Moreover, the projection $\Pi_{\mathcal{M}}$ can be implemented block by block via standard projection onto spectral-norm balls, which is computationally tractable and can be carried out using singular-value decomposition.
\par (2) Predictive online optimization: For the resulting OCO problem with memory, POC uses the available lookahead surrogate losses to update the DAC parameter before applying the control action. Since the surrogate loss $\tilde f_t$ is an expected surrogate loss, we assume access to an exact expected-gradient oracle for $\nabla \tilde f_t(M)$. The algorithm maintains a rolling initialization sequence $\{M_s(0)\}$, where the initial predictions $\tilde f_1,\ldots,\tilde f_{W-1}$ generate $M_2(0),\ldots,M_W(0)$ by OGD, and at each time step $t$ the newly available prediction $\tilde f_{t+W-1}$ extends this sequence via  
\begin{align}
M_{\iota+1}(0)=\Pi_{\mathcal{M}}\left[M_{\iota}(0)-\delta \nabla \tilde{f}_{\iota}\left(M_{\iota}(0)\right)\right],
\end{align}
where $\iota =t+W-1$ and $\delta>0$ is the OGD step-size. The point $M_{t+W}(0)$ provides the fixed right boundary for the current window rather than an optimization variable.
%{\color{blue} Since the disturbance distribution is assumed to be known and stationary, the controller can access an exact expected gradient oracle, enabling the analytical computation of $\nabla \tilde{f}_t(M)$. This ensures that the proximal gradient updates within the lookahead window are strictly deterministic.} 
% For ease of notation, we extend the definitions as follows. Let ${f}_t(\cdot)=0$ for $t\leq 0$ or $t> T$, $u_t=u_0$, $x_t=x_0$, $M_t=M_0$ for $t\leq 0$. In addition, we define $M_t(k)=M_0$ for $t\leq 0$ and $k\geq 0$ when necessary. 
% \par At time step $t$, the algorithm first constructs an OGD warm start for the terminal boundary point of the lookahead window. This update uses only the surrogate losses that are available at time step $t$, i.e. $\tilde{f}_1, \ldots, \tilde{f}_{t+W-1}$ and previously computed iterates. Define $\iota =t+W-1$, then the OGD update is given by 
% \begin{align}
% M_{\iota+1}(0)=\Pi_{\mathcal{M}}\left[M_{\iota}(0)-\delta \nabla \tilde{f}_{\iota}\left(M_{\iota}(0)\right)\right],
% \end{align}
% where the step-size $\delta>0$ and $M_{t+W-1}(0)$ is available from the previous stage $t-1$. 
% The resulting $M_{t+W}(0)$ is then treated as the fixed right boundary for the windowed receding-horizon problem at time step $t$. 
\par Given the active window
\[
\mathbf{M}_{t,W}(k)\triangleq \left(M_t(k),M_{t+1}(k),\ldots,M_{t+W-1}(k)\right)\in \mathcal{M}^W .
\]
With the left boundary $M_{t-1}$ fixed by the previous control step and the right boundary $M_{t+W}(0)$ fixed by the OGD update, define
% The initial vector $\mathbf{M}_{t,W}(0)$ is obtained by shifting the previously stored iterates over the window. Starting from this initialization, the algorithm performs $W$ constrained proximal-gradient refinements on the composite windowed objective. Let
\begin{align}\label{F_R_Def_alg}
F_t(\mathbf{M}_{t,W})&\triangleq \sum_{s=t}^{t+W-1}\tilde f_s(M_s), \notag\\
R_t(\mathbf{M}_{t,W})&\triangleq \sum_{s=t}^{t+W-1}\lambda\|M_s-M_{s-1}\|_F\notag\\
&\quad +\lambda\|M_{t+W}(0)-M_{t+W-1}\|_F ,
\end{align}
where $F_t$ collects the lookahead surrogate losses and $R_t$ is the windowed total-variation regularizer on the DAC policy
parameters. Since $R_t$ is non-smooth, the receding-horizon objective has a composite structure, which motivates the joint proximal-gradient update developed below.

\par Starting from $\mathbf{M}_{t,W}(0)$, POC performs $W$ joint constrained proximal-gradient refinements
\begin{align}
&\mathbf{M}_{t,W}(k)
=\arg\min_{\mathbf{Z}\in\mathcal{M}^{W}}\bigg\{
R_t(\mathbf{Z})
\notag\\
&+\frac{1}{2\eta}
\left\|\mathbf{Z}-\left(\mathbf{M}_{t,W}(k-1)-\eta\nabla F_t\left(\mathbf{M}_{t,W}(k-1)\right)\right)\right\|_F^2
\bigg\},
\end{align}
where $\eta>0$ is the proximal-gradient step-size. This joint proximal-gradient step handles the non-smooth switching regularizer in the receding-horizon objective while simultaneously refining all decision variables in the lookahead window. After $W$ inner proximal iterations,  the first block is selected as $[\mathbf{M}_{t,W}(W)]^{(1)}$, which is then used in the DAC policy to generate the control action. The new state is observed afterward, and $w_t=x_{t+1}-Ax_t-Bu_t$ is recovered for future updates.
% \begin{align}
% M_s(k)=\arg \min _{M \in \mathcal{M}}\left\{R_s(M)+\frac{1}{2 \eta}\left\|M-\hat{M}_s(k-1)\right\|_F^2\right\},
% \end{align}
% where $\hat{M}_s(k-1)\triangleq M_s(k-1)-\eta \nabla \tilde{f}_s\left(M_s(k-1)\right)$, $R_s(M)=\lambda \|M-M_{s-1}(k-1)\|_F+\lambda \|M_{s+1}(k-1)-M\|_F$, the step-size $\eta>0$, and $k=t+W-s$. The proximal gradient descent update is designed to handle the non-smooth regularization term of the receding horizon objective $g_T^{RH}$ in \eqref{rh_obj}. The update is performed iteratively from $t+W-1$ down to $t$, which allows the algorithm to compute the sequence $M_{t+W-1}(1), M_{t+W-2}(2),\cdots,M_t(W)$. Finally, the algorithm outputs $M_t(W)$, which is used to construct the control action at next stage. 
\begin{remark}
	In contrast to the existing work \cite{Zhao2023NonStationaryOnlineLearning}, which relies on meta-base decomposition and develops a two-level algorithmic structure, the proposed algorithm adopts a single-loop update scheme. This streamlined structure simplifies the algorithmic design, reducing computational overhead and improving practical implementability. Moreover, our method explicitly incorporates predicted cost functions into the online decision-making process. In particular, the predictions are treated as exogenous inputs, rather than being generated by the control algorithm itself. While this separation allows for a clean integration of predictive information, the development of a unified framework that jointly performs prediction and online control remains an important direction for future research.
\end{remark}

\section{Policy Regret Analysis}\label{analysis_sec}
This section presents the policy regret analysis of the proposed POC algorithm. We first introduce several standard assumptions commonly adopted in the literature, and then provide the regret analysis.

\begin{assumption}\label{dynamic_ass}
	The matrices governing the system dynamics are bounded, i.e., $\|A\| \leq \kappa_A,\|B\| \leq \kappa_B$. The disturbance at each time step is bounded, i.i.d, and zero-mean with a lower bounded covariance, i.e.,
	\[
	\forall t,\; w_t \sim \mathcal{D}_w,\; \mathbb{E}\left[w_t\right]=0,\;
	\mathbb{E}\left[w_t w_t^{\top}\right] \succeq \sigma^2 I \;\text{ and }\;\left\|w_t\right\| \leq \bar{w}.
	\]
\end{assumption}
\begin{assumption}\label{obj_ass}
	The costs $c_t(x, u)$ are $\alpha$-strongly convex and $l$-smooth. Further, as long as it is guaranteed that $\|x\|,\|u\| \leq D$, it holds that 
	
	$$
	\left\|\nabla_x c_t(x, u)\right\|,\left\|\nabla_u c_t(x, u)\right\| \leq G D .
	$$
	
	%{\color{blue}Then, $\|\nabla \tilde{f}_{t-1}(M_{t-1})\|\leq G$. And $g_T^{RH}(\cdot)$ and $\tilde{f}_t$ is strongly convex and $L$ smooth.}
\end{assumption}
\begin{assumption}\label{DAC_ass}
	The DAC controller $\pi(K,M)$ satisfies
	\begin{itemize}
		\item[(1)] The feedback matrix $K$ is $(\kappa,\gamma)$-strongly stable, i.e., for real constants $\kappa\geq 1$ and $\gamma\in (0,1)$, there exists a complex diagonal matrix $\Lambda$ and a non-singular complex matrix $Q$, such that $A-B K=Q \Lambda Q^{-1}$ and the following conditions are met:
		\begin{itemize}
			\item The spectral norm of $\Lambda$ is strictly smaller than one, i.e., $\|\Lambda\| \leq 1-\gamma$.
			\item The controller and the transformation matrices are bounded, i.e., $\|K\| \leq \kappa$ and $\|Q\|,\left\|Q^{-1}\right\| \leq \kappa$.
		\end{itemize}
		\item[(2)] The disturbance-action parameter $M$ lies in a constrained set $\mathcal{M}$ defined in \eqref{M_set_def}. %$$\mathcal{M}=\{M=(M^{[0]},\cdots, M^{[H-1]})|\|M^{[i-1]}\|\leq \kappa^3\kappa_B(1-\gamma)^i\}.$$
	\end{itemize}
\end{assumption}
\begin{remark}
	These assumptions are standard and widely adopted in the online control and 
	online convex optimization literature \cite{Agarwal2018OnlineControl,agarwal2019logarithmic,survey_online,YANG2025112525}. Assumption~\ref{dynamic_ass} imposes mild regularity conditions on the system dynamics and disturbances. The boundedness of $A$ and $B$ is a standard requirement to ensure the system is well-posed, while the i.i.d., zero-mean disturbance with a lower-bounded covariance is a common statistical condition that captures a wide range of practical noise models \cite{regret_LQ_subGua,ABBASZADEHCHEKAN2023110876}. The gradient boundedness condition in Assumption~\ref{obj_ass} is a natural consequence of the compactness of the state-input domain. 
	Assumption~\ref{DAC_ass} characterizes the structure of the DAC policy class. The $(\kappa, \gamma)$-strong stability of $K$ guarantees closed-loop stability under the base controller, while the constraint set $\mathcal{M}$ on the disturbance-action parameters bounds the magnitude of the DAC coefficients and ensures the tractability of the policy optimization problem. This DAC-based policy parameterization has been widely adopted in prior works \cite{Agarwal2018OnlineControl,agarwal2019logarithmic,Zhao2023NonStationaryOnlineLearning,online_fw_cdc}.
\end{remark}
%\par The next lemma is to highlight the reduction to OCO with memory. It shows that achieving {\color{blue}\textit{low policy regret}} on the memory-based function $f_t$ is sufficient to ensure low regret on the overall dynamical system.
Under these assumptions, the following lemmas provide vital properties for the truncated loss function $f_t$.
% \begin{lemma}[Lemma 5.2, \cite{agarwal2019logarithmic}]\label{f_t_lips}
%     The norm of the gradients of function $f_t$ is bounded by $G_f\leq G D H d \bar{w}\left(H+\frac{2 \kappa_B \kappa^3}{\gamma}\right)$, where $$D \triangleq \frac{\bar{w} \kappa^3\left(1+H \kappa_B^2 \kappa^3\right)}{\gamma\left(1-\kappa^2(1-\gamma)^{H+1}\right)}+\frac{\kappa_B \kappa^3 \bar{w}}{\gamma}.$$
% \end{lemma}

% \begin{lemma}[Lemma 4.2, \cite{agarwal2019logarithmic}]
%      The surrogate  loss function $\tilde{f}_t(M)$ are $\mu$-strongly convex with respect to $M$ where $\mu=\frac{\alpha \sigma^2 \gamma^2}{36 \kappa^{10}}$.
% \end{lemma}
\begin{lemma}
	\label{f_t_lips}
	Under Assumptions \ref{dynamic_ass}, \ref{obj_ass} and \ref{DAC_ass}, the following properties hold:
	\begin{itemize}
		\item \textbf{Lipschitz continuity and gradient bound.} \cite[Lemma 29]{Zhao2023NonStationaryOnlineLearning}
		The truncated loss $f_t$ is $L_f$-coordinate-wise Lipschitz, i.e.,
		\begin{align*}
		&|f_t(M_{t-H-1},\cdots,M_{t-k},\cdots,M_t)\\
		&\quad -f_t(M_{t-H-1},\cdots,\tilde{M}_{t-k},\cdots, M_t)|\\
		&\leq L_f\|M_{t-k}-\tilde{M}_{t-k}\|_F,
		\end{align*}
		where the Lipschitz constant satisfies
		\[L_f\leq 3\sqrt{H}GD\bar{w}\kappa_B\kappa^3.\]
		Moreover, the gradient of the surrogate loss $\tilde{f}_t$ is uniformly bounded as
		\[
		 \|\nabla \tilde{f}_t\|_F  
		\le G_f\triangleq 3Hd^2G\bar{w}\kappa_B\kappa^3\gamma^{-1},
		\]
		where the constant $D$ is defined by
		\begin{align}\label{D_def_eq}
		D \triangleq \frac{\bar{w}\kappa^3\left(1+H\kappa_B^2\kappa^3\right)}
		{\gamma\left(1-\kappa^2(1-\gamma)^{H+1}\right)}
		+\frac{\kappa_B\kappa^3\bar{w}}{\gamma}.
		\end{align}
		
		\item \textbf{Strong convexity.}  \cite[Lemma 4.2]{agarwal2019logarithmic}
		The surrogate loss $\tilde{f}_t(M)$ is $\mu$-strongly convex with respect to $M$, where
		\[
		\mu=\frac{\alpha\sigma^2\gamma^2}{36\kappa^{10}}.
		\]
        % \item \textbf{Smoothness.} 
        % By Assumption \ref{obj_ass} and the linear DAC mapping, the surrogate loss $\tilde{f}_t(M)$ is $L$-smooth over $\mathcal{M}$ with the smoothness constant $L = \mathcal{O}(lD^2)$.
%         Since the original cost $c_t(x,u)$ is $l$-smooth (Assumption III.2) and the truncated state and control actions are linear transformations of the DAC parameter $M$, the expected surrogate loss $\tilde{f}_t(M)$ is also smooth. Specifically, based on the chain rule for the Hessian, the smoothness constant scales with the square of the dynamic mapping bound $D$. Thus, $\tilde{f}_t(M)$ is $L$-smooth with respect to $M$, where $L = \mathcal{O}(l D^2)$. For any $M, M' \in \mathcal{M}$, it holds that:
% $$ \|\nabla \tilde{f}_t(M) - \nabla \tilde{f}_t(M')\|_F \le L \|M - M'\|_F. $$
	\end{itemize}
\end{lemma}
{
Complementing the lower-bound analysis in
\cite[Lemma 4.2 and Appendix F]{agarwal2019logarithmic}, we establish a uniform upper bound on the expected Jacobian Gram matrix, yielding a smoothness constant independent of
the memory horizon.
\begin{lemma}
\label{lem:surrogate-smoothness}
Suppose that Assumptions \ref{dynamic_ass}, \ref{obj_ass} and \ref{DAC_ass} hold. Let
$$
\Sigma_w:=\mathbb{E}[w_tw_t^\top], \qquad \overline{\sigma}_w^2:=\|\Sigma_w\|.$$
% Since $\|w_t\|\le \overline w$ almost surely, it follows that
% \[
% \Sigma_w\preceq \overline{\sigma}_w^2 I
% \preceq \overline w^2 I.
% \]
Then the surrogate loss $\tilde{f}_t(M)$ is $L$-smooth over $\mathcal M$, where
$$L\le l C_{\mathrm{DAC}}^2,$$
and
$$
C_{\mathrm{DAC}}^2:=\overline{\sigma}_w^2\left[\left(\frac{\kappa^2\kappa_B}{\gamma}\right)^2+\left(1+\frac{\kappa^3\kappa_B}{\gamma}\right)^2\right].$$
%In particular, $L$ is independent of the DAC memory horizon $H$.
\end{lemma}
}
\begin{proof}
By the $(\kappa,\gamma)$-strong stability of $K$ in Assumption \ref{DAC_ass}, we have
for every $j\ge 0$,
\begin{equation}
\label{eq:closed-loop-decay}
\|\widetilde A_K^j\|\le\|Q\|\,\|\Lambda\|^j\,\|Q^{-1}\|\le\kappa^2(1-\gamma)^j.
\end{equation}
Consider two DAC parameters
$M=(M^{[0]},\ldots,M^{[H-1]})$ and
$N=(N^{[0]},\ldots,N^{[H-1]})$,
and define
$$\Delta M:=M-N,\qquad\Delta M^{[r]}:=M^{[r]}-N^{[r]}.$$
From the definition of the transfer matrix, the difference between the two truncated states is
\begin{align}
\Delta y_{t+1}&:=y_{t+1}^K(M)-y_{t+1}^K(N)
\notag\\
&=\sum_{r=0}^{H-1}\sum_{j=0}^{H}\widetilde A_K^j B\Delta M^{[r]} w_{t-j-r-1}.
\label{eq:delta-y-double-sum}
\end{align}
% The autonomous term
% \[
% \sum_{i=0}^H\widetilde A_K^i w_{t-i}
% \]
% does not appear in \eqref{eq:delta-y-double-sum}, since it is
% independent of $M$ and therefore cancels when taking the
% difference.
Define $A_j:=\widetilde A_K^jB$. Using \eqref{eq:closed-loop-decay} and $\|B\|\le\kappa_B$, we have
\begin{equation}\label{eq:Aj-bound}
\|A_j\|\le \kappa^2\kappa_B(1-\gamma)^j.
\end{equation}
Let $n=j+r$ and define
$C_n(\Delta M):=\sum_{\substack{0\le r\le H-1\\0\le n-r\le H}}A_{n-r}\Delta M^{[r]},\qquad n=0,\ldots,2H-1$. Then, \eqref{eq:delta-y-double-sum} satisfies
\begin{equation}\label{eq:delta-y-convolution}
\Delta y_{t+1}=\sum_{n=0}^{2H-1} C_n(\Delta M)w_{t-n-1}.
\end{equation}
\par Since the disturbances are independent and zero-mean, the cross terms associated with distinct time indices vanish, which gives
\begin{align}
\mathbb E\|\Delta y_{t+1}\|^2&=\sum_{n=0}^{2H-1}\mathbb E\left[\left\|C_n(\Delta M)w_{t-n-1}\right\|^2\right]\notag\\
&=\sum_{n=0}^{2H-1}\operatorname{tr}\left(C_n(\Delta M)\Sigma_w C_n(\Delta M)^\top \right)\notag\\
&\le \overline{\sigma}_w^2 \sum_{n=0}^{2H-1} \|C_n(\Delta M)\|_F^2.
\label{eq:delta-y-expectation}
\end{align}
Let $a_j:=\|A_j\|$, $m_r:=\|\Delta M^{[r]}\|_F$. By submultiplicativity, we obtain 
$$
\|C_n(\Delta M)\|_F \le \sum_{\substack{0\le r\le H-1\\0\le n-r\le H}} a_{n-r}m_r.
$$
Hence the sequence $\{\|C_n(\Delta M)\|_F\}_n$ is bounded by the discrete convolution of $\{a_j\}_j$ and $\{m_r\}_r$. Using Young's convolution inequality gives
\begin{align}
\left(\sum_{n=0}^{2H-1} \|C_n(\Delta M)\|_F^2 \right)^{1/2}&\le\left(\sum_{j=0}^H a_j\right)\left(\sum_{r=0}^{H-1}m_r^2\right)^{1/2}\notag\\
&=\left(\sum_{j=0}^H\|A_j\|\right)\|\Delta M\|_F.
\label{eq:young-convolution}
\end{align}
Using \eqref{eq:Aj-bound}, we have
\begin{align}
\sum_{j=0}^H\|A_j\|&\le\kappa^2\kappa_B\sum_{j=0}^H(1-\gamma)^j\notag\\
&\le\kappa^2\kappa_B \sum_{j=0}^\infty(1-\gamma)^j
=\frac{\kappa^2\kappa_B}{\gamma}.
\label{eq:geometric-gain}
\end{align}
Combining
\eqref{eq:delta-y-expectation}, \eqref{eq:young-convolution} and \eqref{eq:geometric-gain},
we obtain
\begin{equation}\label{eq:state-gain}
\mathbb E\|\Delta y_{t+1}\|^2\le C_y^2\|\Delta M\|_F^2,
\end{equation}
where $C_y:=\frac{\overline{\sigma}_w\kappa^2\kappa_B}{\gamma}$.

We next bound the variation of the truncated control input. By its definition,
\begin{align}
\Delta v_{t+1}&:=v_{t+1}^K(M)-v_{t+1}^K(N)\notag\\
&=-K\Delta y_{t+1}+\sum_{r=0}^{H-1}\Delta M^{[r]}w_{t-r}.
\label{eq:delta-v}
\end{align}
Let $d_{t+1}(\Delta M):=\sum_{r=0}^{H-1} \Delta M^{[r]}w_{t-r}$.
The independence and zero-mean disturbances also imply
\begin{align}
\mathbb E\|d_{t+1}(\Delta M)\|^2&=\sum_{r=0}^{H-1}\operatorname{tr} \left(\Delta M^{[r]}\Sigma_w \Delta M^{[r]\top} \right)\notag\\
&\le \overline{\sigma}_w^2 \sum_{r=0}^{H-1} \|\Delta M^{[r]}\|_F^2 \notag\\
&\leq \overline{\sigma}_w^2 \|\Delta M\|_F^2.
\label{eq:direct-control-gain}
\end{align}

Applying Minkowski's inequality in the Hilbert space $L_2$ to \eqref{eq:delta-v}, and using $\|K\|\le\kappa$, \eqref{eq:state-gain}, and \eqref{eq:direct-control-gain}, yields
\begin{align}
&\left( \mathbb E\|\Delta v_{t+1}\|^2 \right)^{1/2}\notag\\
\le& \|K\| \left( \mathbb E\|\Delta y_{t+1}\|^2 \right)^{1/2} +\left( \mathbb E\|d_{t+1}(\Delta M)\|^2 \right)^{1/2} \notag\\
\le& \left( \kappa C_y+\overline{\sigma}_w\right) \|\Delta M\|_F.
\end{align}
Define
\begin{equation}
\label{eq:Cv-definition}
C_v := \kappa C_y+\overline{\sigma}_w=\overline{\sigma}_w\left(1+\frac{\kappa^3\kappa_B}{\gamma}\right).
\end{equation}
It follows that
\begin{equation}
\label{eq:control-gain}
\mathbb E\|\Delta v_{t+1}\|^2\le C_v^2\|\Delta M\|_F^2.
\end{equation}

Combining \eqref{eq:state-gain} and
\eqref{eq:control-gain}, we obtain
\begin{align}
\mathbb E \left\|\begin{bmatrix}\Delta y_{t+1}\\\Delta v_{t+1}\end{bmatrix}\right\|^2
&=\mathbb E\|\Delta y_{t+1}\|^2+\mathbb E\|\Delta v_{t+1}\|^2\notag\\
&\le \left(C_y^2+C_v^2\right) \|\Delta M\|_F^2 \notag\\
&=C_{\mathrm{DAC}}^2 \|\Delta M\|_F^2,
\label{eq:joint-DAC-gain}
\end{align}
where
\begin{align}
C_{\mathrm{DAC}}^2 &:=C_y^2+C_v^2\notag\\
&=\overline{\sigma}_w^2\left[\left(\frac{\kappa^2\kappa_B}{\gamma}\right)^2+\left(1+\frac{\kappa^3\kappa_B}{\gamma}\right)^2\right].
\label{eq:CDAC-definition}
\end{align}

We now use \eqref{eq:joint-DAC-gain} to establish the smoothness of $\widetilde f_t$. For a fixed disturbance realization, define the affine mapping
$$
\mathcal Z_t(M;w):=\begin{bmatrix}
y_t^K(M;w)\\
v_t^K(M;w)
\end{bmatrix}=b_t(w)+\mathcal G_t(w)M,
$$
where $\mathcal G_t(w)$ is the linear part of the DAC mapping. The \eqref{eq:joint-DAC-gain} implies that, for every parameter direction $V$,
\begin{equation}
\label{eq:G-mean-square-bound}
\mathbb E\|\mathcal G_t(w)V\|^2\le C_{\mathrm{DAC}}^2\|V\|_F^2.
\end{equation}

Because the disturbance is bounded, $\mathcal M$ is compact, and the gradients of $c_t$ are bounded on the induced state-action domain, differentiation and expectation can be interchanged. Hence, we have
\begin{equation}
\label{eq:surrogate-gradient}
\nabla\widetilde f_t(M)=\mathbb E\left[\mathcal G_t(w)^* \nabla c_t\bigl(\mathcal Z_t(M;w)\bigr)\right],
\end{equation}
where $\mathcal G_t(w)^*$ denotes the adjoint operator of $\mathcal G_t(w)$ under the product Frobenius inner product.

\par Let $U$ be any parameter direction with $\|U\|_F=1$. Using \eqref{eq:surrogate-gradient}, the adjoint identity, and the $l$-smoothness of $c_t$, we have
\begin{align}
&\left|\left\langle U, \nabla\widetilde f_t(M) - \nabla\widetilde f_t(N) \right\rangle \right| \notag\\
&=\left|\mathbb E \left[ \left\langle \mathcal G_t(w)U, \nabla c_t\bigl(\mathcal Z_t(M;w)\bigr) - \nabla c_t\bigl(\mathcal Z_t(N;w)\bigr) \right\rangle\right]\right|\notag\\
&\le l \mathbb E \left[ \|\mathcal G_t(w)U\| \|\mathcal G_t(w)\Delta M\| \right].
\label{eq:gradient-difference-bound}
\end{align}
Applying the Cauchy-Schwarz inequality in expectation and using \eqref{eq:G-mean-square-bound}, we obtain
\begin{align}
&\mathbb E \left[ \|\mathcal G_t(w)U\| \|\mathcal G_t(w)\Delta M\|\right]\notag\\
&\le \left(\mathbb E\|\mathcal G_t(w)U\|^2\right)^{1/2}\left(\mathbb E\|\mathcal G_t(w)\Delta M\|^2\right)^{1/2}\notag\\
&\le C_{\mathrm{DAC}}^2 \|U\|_F \|\Delta M\|_F = C_{\mathrm{DAC}}^2 \|M-N\|_F.
\end{align}
Substituting this estimate into \eqref{eq:gradient-difference-bound} gives
$$ \left| \left\langle U, \nabla\widetilde f_t(M) - \nabla\widetilde f_t(N) \right\rangle \right| \le lC_{\mathrm{DAC}}^2 \|M-N\|_F.$$
Since $U$ is an arbitrary unit-norm direction, we have
$$\left\| \nabla\widetilde f_t(M) -\nabla\widetilde f_t(N) \right\|_F \le lC_{\mathrm{DAC}}^2 \|M-N\|_F.$$
Therefore, $\widetilde f_t$ is $L$-smooth with $L\le lC_{\mathrm{DAC}}^2$.

For the initial stages $t\le H$, the truncated state and control contain only a subset of the terms used above. Therefore, the same upper bound continues to hold.
\end{proof}
\par In section \ref{red_oco_sec}, we established a reduction from online control to online convex optimization with memory by adopting the DAC parameterization and introducing truncated loss functions. Building on this reduction, the next lemma quantifies the truncation error and characterizes how accurately the truncated objective approximates the original control cost. %relates the cost of $f_t\left(M_{t-H-1: t}\right)$ with the actual cost $c_t\left(x_t^K\left(M_{0: t-1}\right), u_t^K\left(M_{0: t}\right)\right)$. [Theorem 5.3, \cite{Agarwal2018OnlineControl}]
\begin{lemma}\cite[Theorem 5.3]{Agarwal2018OnlineControl}\label{trunc_approx_err}
	For any $(\kappa, \gamma)$-strongly stable $K$, under Assumptions \ref{dynamic_ass}, \ref{obj_ass} and \ref{DAC_ass}, we have that
	\begin{align*}
	&\left|\sum_{t=1}^T f_t\left(M_{t-H-1: t}\right)-\sum_{t=1}^T c_t\left(x_t^K\left(M_{0: t-1}\right), u_t^K\left(M_{0: t}\right)\right) \right|\\
	&\leq 2 T G D^2 \kappa^3(1-\gamma)^{H+1},
	\end{align*}
	where $D$ is defined in \eqref{D_def_eq}.
	%where
	% $$
	% D \triangleq \frac{W \kappa^3\left(1+H \kappa_B \tau\right)}{\gamma\left(1-\kappa^2(1-\gamma)^{H+1}\right)}+\frac{\tau W}{\gamma} .
	% $$
\end{lemma}
%     Let the dynamical system satisfy Assumption 2.1 and let $K$ be any $(\kappa, \gamma)$-diagonal strongly stable matrix. Consider a sequence of loss functions $c_t(x, u)$ satisfying Assumption 2.2 and a sequence of policies $M_0 \ldots M_T$ satisfying

% $$
% \text { PolicyRegret }=\sum_{t=0}^T f_t\left(M_{t-H-1: t}\right)-\min _{M \in \mathcal{M}} \sum_{t=0}^T f_t(M) \leq R(T)
% $$
% for some function $R(T)$ and $f_t$ as defined in Definition 3.4. Let $A$ be an online algorithm that plays the non-stationary controller sequence $\left\{M_0, \ldots M_T\right\}$. Then as long as $H$ is chosen to be larger than $\gamma^{-1} \log \left(T \kappa^2\right)$ we have that
%\subsection{Convergence Analysis of Receding Horizon Optimization}
Building on the truncation-approximation result above, we now analyze the optimization dynamics of the proposed POC algorithm. Specifically, {POC adopts a receding-horizon scheme over the lookahead window, with objective $g_T^{RH}(\mathbf{M})=\sum_{t=1}^{T} \tilde{f}_{t}(M_t)+\sum_{t=1}^{T}\lambda  \|M_{t}-M_{t-1}\|_F$.} This objective consists of strongly convex stage losses and a non-smooth switching-cost regularizer, which together capture both performance and policy variation. To analyze the convergence properties, we first establish a linear convergence guarantee for the constrained proximal-gradient updates used within the horizon optimization, and then bound the initialization error induced by the lookahead-based OGD step. These two ingredients are subsequently combined to derive the main convergence result. 

{
We first establish a convergence guarantee for the receding-horizon objective used in the POC inner loop. For each stage $t$, the algorithm considers the finite window
\[
\mathbf M_{t,W}=(M_t,\ldots,M_{t+W-1}),
\]
and performs $W$ proximal-gradient steps on the receding-horizon objective
\begin{align}\label{g_tw_rh}
g_{t,W}^{RH}(\mathbf M_{t,W})	\triangleq F_t(\mathbf M_{t,W})+R_t(\mathbf M_{t,W}),
\end{align}
with $F_t$ and $R_t$ in \eqref{F_R_Def_alg}.
% where
% \begin{align*}
% F_t(\mathbf M_{t,W})&\triangleq \sum_{s=t}^{t+W-1}\tilde f_s(M_s),\\
% R_t(\mathbf M_{t,W})&\triangleq \lambda\sum_{s=t}^{t+W-1}\|M_s-M_{s-1}\|_F
% \\
% &+\lambda\|M_{t+W}^{bd}-M_{t+W-1}\|_F.
% \end{align*}
% Here $M_{t-1}$ and $M_{t+W}^{bd}\triangleq M_{t+W}(0)$ are fixed boundary terms. The update used in Algorithm \ref{POC_alg} is
% \begin{align}
% &\mathbf M_{t,W}(k)
% =\arg\min_{\mathbf Z\in\mathcal M^{W}}\bigg\{
% R_t(\mathbf Z)
% \notag\\
% &+\frac{1}{2\eta}\Big\|\mathbf{Z}-\big(\mathbf M_{t,W}(k-1)-\eta\nabla F_t(\mathbf M_{t,W}(k-1))\big)\Big\|_F^2
% \bigg\},
% \end{align}
% which is exactly the proximal-gradient step for the composite problem $g_{t,W}^{RH}=F_t+R_t$ over the convex feasible set $\mathcal M^{W}$.

\begin{proposition}[Linear Convergence of Proximal Gradient]
\label{prop3}
Under Assumptions \ref{dynamic_ass}, \ref{obj_ass}, and \ref{DAC_ass}, let the proximal step-size satisfy $0 <\eta \le \frac{1}{L}$. Consider the windowed receding-horizon objective function $g_{t,W}^{RH}$ in \eqref{g_tw_rh} defined over the closed convex set $\mathcal{M}^W$. % where $F_t$ is $L$-smooth and $\mu$-strongly convex, and $R_t$ is proper, closed, and convex (potentially non-smooth). 
Let $\mathbf{M}_{t,W}^* \in \mathcal{M}^W$ be the unique global minimizer of $g_{t,W}^{RH}$ over $\mathcal{M}^W$. Then, the sequence of iterates $\{\mathbf{M}_{t,W}(k)\}_{k \ge 0}$ generated by the proximal gradient update  satisfies
\begin{equation}
\|\mathbf{M}_{t,W}(k) - \mathbf{M}_{t,W}^*\|_F \le \rho^k \|\mathbf{M}_{t,W}(0) - \mathbf{M}_{t,W}^*\|_F,
\end{equation}
where the contraction factor is given by $\rho = \sqrt{1 - \mu\eta} < 1$.
\end{proposition}

\begin{proof}
For presentation conciseness, we drop the indices $t$ and $W$, denoting $\mathbf{M}_{t,W}(k)$ as $\mathbf{M}(k)$, $\mathbf{M}_{t,W}^*$ as $\mathbf{M}^*$, $F_t$ as $F$, and $R_t$ as $R$. To capture the constraint set $\mathcal{M}^W$, we incorporate the indicator function $\mathbb{I}_{\mathcal{M}^W}(\mathbf{M})$ into the non-smooth term by defining $h(\mathbf{M}) \triangleq R(\mathbf{M}) + \mathbb{I}_{\mathcal{M}^W}(\mathbf{M})$. Since $\mathcal{M}^W$ is closed and convex, $h(\mathbf{M})$ remains a proper, closed, and convex function.
\par Under this formulation, the proximal gradient update step \eqref{inn_prox_up} can be expressed via the composite proximal mapping operator $\mathcal{T}_\eta: \mathcal{M}^W \rightarrow \mathcal{M}^W$:
\begin{align*}
\mathbf{M}(k) = & \mathcal{T}_\eta(\mathbf{M}(k-1)) \\
\triangleq & \text{prox}_{\eta h} \left( \mathbf{M}(k-1) - \eta \nabla F(\mathbf{M}(k-1)) \right),
\end{align*}
where the proximal operator is defined as $\text{prox}_{\eta h}(\mathbf{Z}) \triangleq \arg\min_{\mathbf{M}} \{ h(\mathbf{M}) + \frac{1}{2\eta} \|\mathbf{M} - \mathbf{Z}\|_F^2 \}$. Since $\mathbf{M}^*$ is the unique optimal solution over the feasible set, it constitutes a fixed point of the operator $\mathcal{T}_\eta$, satisfying
\begin{equation}
\mathbf{M}^* = \mathcal{T}_\eta(\mathbf{M}^*) = \text{prox}_{\eta h} \left( \mathbf{M}^* - \eta \nabla F(\mathbf{M}^*) \right).
\end{equation}
% By the fundamental property of proximal mappings for any proper closed convex function, the operator $\text{prox}_{\eta h}$ is non-expansive. 
Applying the non-expansiveness property of proximal mapping operators yields
\begin{align}
&\|\mathbf{M}(k) - \mathbf{M}^*\|_F \notag\\
=& \|\text{prox}_{\eta h} \left( \mathbf{M}(k-1) - \eta \nabla F(\mathbf{M}(k-1)) \right) \notag\\
&\quad - \text{prox}_{\eta h} \left( \mathbf{M}^* - \eta \nabla F(\mathbf{M}^*) \right)\|_F \notag \\
\le& \|\left( \mathbf{M}(k-1) \!- \eta \nabla F(\mathbf{M}(k-1)) \right)\! -\! \left( \mathbf{M}^* \!- \eta \nabla F(\mathbf{M}^*) \right)\|_F. \label{eq:prox_non_expansive}
\end{align}
\par Next, we analyze the contraction property of the gradient descent step on the smooth part. Let $\mathbf{\Delta} \triangleq \mathbf{M}(k-1) - \mathbf{M}^*$. Expanding the squared Frobenius norm of the right-hand side of \eqref{eq:prox_non_expansive} gives
\begin{align}
& \|\left( \mathbf{M}(k-1) - \eta \nabla F(\mathbf{M}(k-1)) \right) - \left( \mathbf{M}^* - \eta \nabla F(\mathbf{M}^*) \right)\|_F^2 \nonumber \\
=& \|\mathbf{\Delta}\|_F^2 - 2\eta \langle \mathbf{\Delta}, \nabla F(\mathbf{M}(k-1)) - \nabla F(\mathbf{M}^*) \rangle\notag\\
&+ \eta^2 \|\nabla F(\mathbf{M}(k-1)) - \nabla F(\mathbf{M}^*)\|_F^2. \label{eq:squared_expansion}
\end{align}
Since $F$ is $L$-smooth and $\mu$-strongly convex, its gradient satisfies the standard co-coercivity property
\begin{align*}
&\langle \mathbf{\Delta}, \nabla F(\mathbf{M}(k-1)) - \nabla F(\mathbf{M}^*) \rangle \\
\ge& \frac{\mu L}{\mu + L} \|\mathbf{\Delta}\|_F^2 + \frac{1}{\mu + L} \|\nabla F(\mathbf{M}(k-1)) - \nabla F(\mathbf{M}^*)\|_F^2.
\end{align*}
Substituting this bound into \eqref{eq:squared_expansion} yields
\begin{align}
& \|( \mathbf{M}(k-1) \!- \eta \nabla F(\mathbf{M}(k-1)) ) - ( \mathbf{M}^* \! - \eta \nabla F(\mathbf{M}^*) )\|_F^2 \nonumber \\
\le& \left( 1 - \frac{2\eta \mu L}{\mu + L} \right) \|\mathbf{\Delta}\|_F^2 \notag\\
&+ \eta \left( \eta - \frac{2}{\mu + L} \right) \|\nabla F(\mathbf{M}(k-1)) - \nabla F(\mathbf{M}^*)\|_F^2. \label{co_coercive_substituted}
\end{align}
Since $\eta \le 1/L$ and  $L \ge \mu > 0$, we have $\eta \le \frac{1}{L} \le \frac{2}{\mu + L}$, which directly guarantees that the coefficient of the second term is non-positive, i.e., $\eta ( \eta - \frac{2}{\mu + L} ) \le 0$. Dropping this non-positive term simplifies the bound to
\begin{align}\label{sqr_up_fi}
&\|\left( \mathbf{M}(k-1) - \eta \nabla F(\mathbf{M}(k-1)) \right) - \left( \mathbf{M}^* - \eta \nabla F(\mathbf{M}^*) \right)\|_F^2 \notag \\
\le &\left( 1 - \frac{2\eta \mu L}{\mu + L} \right) \|\mathbf{\Delta}\|_F^2.
\end{align}
Furthermore, due to $L \ge \mu>0$, it implies $\frac{L}{\mu + L} \ge \frac{1}{2}$ and $\frac{2\mu L}{\mu + L} \ge \mu$. Applying this relationship to the contraction coefficient yields
\begin{equation}
1 - \frac{2\eta \mu L}{\mu + L} \le 1 - \mu \eta.
\end{equation}
Taking the square root on both sides of \eqref{sqr_up_fi} and substituting it into \eqref{eq:prox_non_expansive} lead to 
\begin{align*}
\|\mathbf{M}(k) - \mathbf{M}^*\|_F \le& \sqrt{1 - \mu \eta} \|\mathbf{M}(k-1) - \mathbf{M}^*\|_F \\
=& \rho \|\mathbf{M}(k-1) - \mathbf{M}^*\|_F.
\end{align*}
Applying it recursively for $k$ iterations completes the proof.
% \begin{equation}
% \label{eq:final_prop_bound}
% \|\mathbf{M}(k) - \mathbf{M}^*\|_F \le \rho^k \|\mathbf{M}(0) - \mathbf{M}^*\|_F.
% \end{equation}
\end{proof}}

To further exploit the lookahead information available at time step $t$, we next analyze the online gradient descent update used to initialize the POC algorithm. The following proposition establishes a regret bound for the OGD update with respect to the objective $g_T^{RH}$. 

\begin{proposition}[Dynamic Regret of the OGD Initialization]
\label{proposition:ogd_init}
Under Assumptions \ref{dynamic_ass}, \ref{obj_ass}, and \ref{DAC_ass}, let the decision set $\mathcal{M}$ have a bounded diameter $D_M = \sup_{M, N \in \mathcal{M}} \|M - N\|_F$. The regret of OGD update \eqref{OGD_up} with stepsize $\delta = \frac{1}{L}$ against the {global comparator sequence} $\mathbf{M}_{RH}^*$ satisfies
\begin{equation}
g_T^{RH}(\mathbf{M}(0)) - g_T^{RH}(\mathbf{M}_{RH}^*) \le B_{OGD},
\end{equation}
where the bound $B_{OGD}$ is given by
\begin{equation}
B_{OGD} = \frac{D_M^2}{2\delta} + \delta T G_f \left( \frac{G_f}{2} + \lambda \right) + \frac{D_M}{\delta} P_T^{RH},
\end{equation}
and $P_T^{RH} = \sum_{t=1}^T \|M_{t,RH}^* - M_{t-1,RH}^*\|_F$ is the path-length of the global comparator sequence.
\end{proposition}

\begin{proof}
For notational conciseness, we denote the OGD iterates $\mathbf{M}(0)$ as $M_t$ and the global comparator sequence $M_{t,RH}^*$ as $M_t^*$ throughout this proof. 
The regret of the OGD sequence evaluated on the regularized objective is given by
\begin{align}
\label{eq:ogd_decomp}
&g_T^{RH}(\mathbf{M}(0)) - g_T^{RH}(\mathbf{M}_{RH}^*) \notag\\
= &\sum_{t=1}^T \left[ \tilde{f}_t(M_t) - \tilde{f}_t(M_t^*) \right] + \lambda \sum_{t=1}^T \|M_t - M_{t-1}\|_F \notag\\
&- \lambda \sum_{t=1}^T \|M_t^* - M_{t-1}^*\|_F.
\end{align}

We first bound the stage cost difference. The OGD update is given by $M_{t+1} = \Pi_{\mathcal{M}} \left[ M_t - \delta \nabla \tilde{f}_t(M_t) \right]$. By the non-expansiveness of the projection operator and the standard expansion of the Euclidean norm, for any sequence $M_t^*$, we have
\begin{align}
\|M_{t+1} - M_t^*\|_F^2 \le& \|M_t - \delta \nabla \tilde{f}_t(M_t) - M_t^*\|_F^2 \nonumber \\
=& \|M_t - M_t^*\|_F^2 \!- 2\delta \langle \nabla \tilde{f}_t(M_t), M_t \!- M_t^* \rangle\notag\\
&+ \delta^2 \|\nabla \tilde{f}_t(M_t)\|_F^2.
\end{align}
Since $\tilde{f}_t$ is $\mu$-strongly convex, it satisfies the lower bound $\langle \nabla \tilde{f}_t(M_t), M_t - M_t^* \rangle \ge \tilde{f}_t(M_t) - \tilde{f}_t(M_t^*) + \frac{\mu}{2}\|M_t - M_t^*\|_F^2$. Substituting this into the above inequality and rearranging the terms, we bound the single-stage cost difference as
\begin{align*}
&\tilde{f}_t(M_t) - \tilde{f}_t(M_t^*) \\
\le &\frac{1-\mu\delta}{2\delta} \|M_t - M_t^*\|_F^2 - \frac{1}{2\delta} \|M_{t+1} - M_t^*\|_F^2 + \frac{\delta}{2} G_f^2,
\end{align*}
where we use the gradient bound $\|\nabla \tilde{f}_t(M_t)\|_F \le G_f$ from Lemma \ref{f_t_lips}. Summing this inequality over $t=1, \dots, T$ yields
\begin{align*}
&\sum_{t=1}^T \left[ \tilde{f}_t(M_t) - \tilde{f}_t(M_t^*) \right] \\
\le &\frac{1}{2\delta} \sum_{t=1}^T \left( \|M_t - M_t^*\|_F^2 - \|M_{t+1} - M_t^*\|_F^2 \right) + \frac{\delta T}{2} G_f^2 \nonumber \\
\le &\frac{1}{2\delta} \|M_1 - M_1^*\|_F^2 + \frac{\delta T}{2} G_f^2\\
&+ \frac{1}{2\delta} \sum_{t=2}^T \left( \|M_t - M_t^*\|_F^2 - \|M_t - M_{t-1}^*\|_F^2 \right).
\end{align*}
To handle the telescoping term, we apply the difference of squares and the triangle inequality $\|M_t - M_{t-1}^*\|_F \le \|M_t - M_t^*\|_F + \|M_t^* - M_{t-1}^*\|_F$, obtaining
\begin{align}
&\|M_t - M_t^*\|_F^2 - \|M_t - M_{t-1}^*\|_F^2 \notag\\
= &\left( \|M_t - M_t^*\|_F - \|M_t - M_{t-1}^*\|_F \right) \times\notag\\
&\quad \left( \|M_t - M_t^*\|_F + \|M_t - M_{t-1}^*\|_F \right) \notag \\
\le &\|M_t^* - M_{t-1}^*\|_F \cdot (2D_M),
\end{align}
where $D_M$ is the diameter of the decision set $\mathcal{M}$. Therefore, the sum of the stage cost differences is bounded by:
\begin{align}
\label{eq:ogd_stage_bound}
&\sum_{t=1}^T \left[ \tilde{f}_t(M_t) - \tilde{f}_t(M_t^*) \right] \notag\\
\le &\frac{D_M^2}{2\delta} + \frac{D_M}{\delta} \sum_{t=2}^T \|M_t^* - M_{t-1}^*\|_F + \frac{\delta T}{2} G_f^2 \notag\\
\le &\frac{D_M^2}{2\delta} + \frac{D_M}{\delta} P_T^{RH} + \frac{\delta T}{2} G_f^2.
\end{align}

Next, we bound the switching cost of the OGD sequence, $\sum_{t=1}^T \|M_t - M_{t-1}\|_F$. For $t=2,\cdots,T$, the non-expansiveness of the projection operator and the gradient bound imply
\begin{align*}
&\|M_t - M_{t-1}\|_F \\
= &\left\| \Pi_{\mathcal{M}} \left[ M_{t-1} - \delta \nabla \tilde{f}_{t-1}(M_{t-1}) \right] - \Pi_{\mathcal{M}} \left[ M_{t-1} \right] \right\|_F \\
\le& \delta \|\nabla \tilde{f}_{t-1}(M_{t-1})\|_F \le \delta G_f.
\end{align*}
Together with $\|M_1(0)-M_0\|_F=0$, we obtain
\begin{equation}
\label{eq:ogd_switch_bound}
\lambda \sum_{t=1}^T \|M_t - M_{t-1}\|_F \le \lambda \delta T G_f.
\end{equation}

Finally, substituting the stage cost bound \eqref{eq:ogd_stage_bound} and the switching cost bound \eqref{eq:ogd_switch_bound} back into the regret decomposition \eqref{eq:ogd_decomp}, and dropping the non-positive term $-\lambda \sum \|M_t^* - M_{t-1}^*\|_F$, we obtain
\begin{align*}
&g_T^{RH}(\mathbf{M}(0)) - g_T^{RH}(\mathbf{M}_{RH}^*)\\
\le& \frac{D_M^2}{2\delta} + \frac{D_M}{\delta} P_T^{RH} + \frac{\delta T}{2} G_f^2 + \lambda \delta T G_f.
\end{align*}
Grouping the constants yields the desired bound $B_{OGD}$, completing the proof.
\end{proof}

Next, we establish the following bound on the suboptimality of the regularized objective $g_T^{RH}$. This result quantifies the combined effect of the proximal-gradient updates and the OGD-based initialization, and will serve as a key ingredient in the dynamic policy regret analysis. { We define the local-to-global mismatch error as $\mathcal{E}_{T, W}:=\sum_{t=1}^T\left\|M_{t, W}^{*(1)}-M_{t, R H}^*\right\|_F $ and $C_g:=G_f+2\lambda$.
}
{
\begin{lemma}\label{gRH_lem}
Under Assumptions \ref{dynamic_ass}, \ref{obj_ass} and \ref{DAC_ass}, let the step-sizes be $\delta = \frac{1}{L}$ and $0 < \eta \le \frac{1}{L}$. Let $\mathbf{M} = (M_1, \dots, M_T)$ be the executed policy sequence generated by the POC algorithm, and $\mathbf{M}_{RH}^* = (M_{1,RH}^*, \dots, M_{T,RH}^*)$ be the global minimizer of $g_T^{RH}$. Then, the suboptimality of the regularized objective satisfies
\begin{align*}
    &g_T^{RH}(\mathbf{M}) - g_T^{RH}(\mathbf{M}_{RH}^*) \\
    \le& C_g\mathcal{O} \left( \rho^W \sqrt{W T  B_{OGD}}+ \rho^W T\sqrt{W} + \mathcal{E}_{T, W} \right)
\end{align*}
where $\rho = \sqrt{1-\mu\eta} < 1$, and $B_{OGD}$ is the upper bound of the OGD initialization from Proposition \ref{proposition:ogd_init}.
\end{lemma}
\begin{proof}
%We first bound the suboptimality of the global objective by the distance in the decision space. 
For any $\mathbf M,\mathbf N\in\mathcal M^T$, the gradient bound $\|\nabla\tilde f_t(M)\|_F\le G_f$ and the reverse triangle inequality for the Frobenius norm give
\begin{align*}
&g_T^{\mathrm{RH}}(\mathbf M) -g_T^{\mathrm{RH}}(\mathbf N) \\
\le& (G_f+2\lambda) \sum_{t=1}^T\|M_t-N_t\|_F \\
=& C_g\sum_{t=1}^T\|M_t-N_t\|_F .
\end{align*}
Taking $\mathbf N=\mathbf M_{\mathrm{RH}}^\star$ gives
\begin{align*}
    g_T^{\mathrm{RH}}(\mathbf M) -g_T^{\mathrm{RH}}(\mathbf M_{\mathrm{RH}}^\star)\le C_g\sum_{t=1}^T \|M_t-M_{t,\mathrm{RH}}^\star\|_F.
\end{align*}
% Due to the bounded gradients of the stage costs $\tilde{f}_t$ and the Lipschitz continuity of the Frobenius norm, the composite objective $g_T^{RH}$ is globally $C_g$-Lipschitz continuous on the bounded domain $\mathcal{M}^T$. Thus, we have

% $$g_T^{RH}(\mathbf{M}) - g_T^{RH}(\mathbf{M}_{RH}^*) \le C_g \sum_{t=1}^T \|M_t - M_{t,RH}^*\|_F$$

For each time step $t$, we introduce the exact minimizer of the local windowed objective denoted as $M_{t,W}^{*(1)}$, which represents the first block of $\mathbf{M}_{t,W}^*$, and then apply the triangle inequality to yield 
\begin{align}\label{dis_ineq_star}
    &\|M_t - M_{t,RH}^*\|_F \notag\\
    \le &\underbrace{\|M_t - M_{t,W}^{*(1)}\|_F}_{\text{optimization error}} + \underbrace{\|M_{t,W}^{*(1)} - M_{t,RH}^*\|_F}_{\text{local-to-global mismatch}}
\end{align}

At first, we bound the optimization error via Proposition \ref{prop3}.
By Proposition \ref{prop3}, the joint proximal-gradient iterates contract linearly toward the local minimizer over the active window, and hence, after $W$ iterations,

$$\|\mathbf{M}_{t,W}(W) - \mathbf{M}_{t,W}^*\|_F \le \rho^W \|\mathbf{M}_{t,W}(0) - \mathbf{M}_{t,W}^*\|_F,$$
where $\rho = \sqrt{1-\mu\eta}$. Then, the optimization error for the first block satisfies
\begin{align*}
    \|[M_{t,W}(W)]^{(1)} - M_{t,W}^{*(1)}\|_F \le &\|\mathbf{M}_{t,W}(W) - \mathbf{M}_{t,W}^*\|_F \\
    \le & \rho^W \|\mathbf{M}_{t,W}(0) - \mathbf{M}_{t,W}^*\|_F
\end{align*}

Using the triangle inequality again to introduce the windowed restriction of the global minimizer $\mathbf{M}_{t,W}^{RH,*}=(M_{t,RH}^*, M_{t+1,RH}^*,\cdots,M_{t+W-1,RH})\in \mathcal{M}^W$, we have
\begin{align}\label{step_1}
    &\|\mathbf{M}_{t,W}(0) - \mathbf{M}_{t,W}^*\|_F \notag\\
    \le& \|\mathbf{M}_{t,W}(0) - \mathbf{M}_{t,W}^{RH,*}\|_F + \|\mathbf{M}_{t,W}^{RH,*} - \mathbf{M}_{t,W}^*\|_F.
\end{align}
% \par Then, we bound the truncation error in \eqref{dis_ineq_star}
% % \begin{align}\label{step_2_1}
% %     \|M_{t,W}^{*(1)} - M_{t,RH}^*\|_F \le C \xi^W.
% % \end{align}
% {\color{blue}
% \begin{align}\label{step_2_1}
%     \|M_{t,W}^{*(1)} - M_{t,RH}^*\|_F \le \|\mathbf{M}_{t,W}^*-\mathbf{M}_{t,RH}^*\|_F\leq D_M\sqrt{W}.
% \end{align}
% }
For the full-window term appearing in the optimization-error bound \eqref{step_1}, we use the boundedness of the feasible set as
\begin{align}\label{step_2_2}
    \|\mathbf{M}_{t,W}^{RH,*}-\mathbf M_{t,W}^{*}\|_F
\le D_{M}\sqrt W.
\end{align}
\par Summing the optimization error over $t=1, \dots, T$, we need to bound the initial distance provided by the OGD oracle. 
The sum of the windowed distances in \eqref{step_1} can be bounded by the global sequence distance as
\begin{align*}
&\sum_{t=1}^T \|\mathbf{M}_{t,W}(0) - \mathbf{M}_{t,W}^{RH,*}\|_F\\
&\leq \sqrt{T}\left(\sum_{t=1}^T
\|\mathbf{M}_{t,W}(0)-\mathbf{M}_{t,W}^{RH,*}\|_F^2\right)^{1/2}\\
&\leq \sqrt{TW}\left(\sum_{t=1}^T
\|M_t(0)-M_{t,RH}^*\|_F^2\right)^{1/2},
\end{align*}
where the last inequality uses the fact that each block appears in at most $W$ active windows.
Since $g_T^{RH}$ is $\mu$-strongly convex over $\mathcal{M}^T$, we can bound the distance to the objective suboptimality using Proposition \ref{proposition:ogd_init} as follows
\begin{align*}
    \sum_{t=1}^T \|M_t(0) - M_{t,RH}^*\|_F^2 =& \|\mathbf{M}(0) - \mathbf{M}_{RH}^*\|_F^2 \\
    \le& \frac{2}{\mu} \left[ g_T^{RH}(\mathbf{M}(0)) - g_T^{RH}(\mathbf{M}_{RH}^*) \right] \\
    \le& \frac{2}{\mu} B_{OGD}.
\end{align*}
Therefore,
\begin{align}\label{step_3}
    \sum_{t=1}^T \|\mathbf{M}_{t,W}(0) - \mathbf{M}_{t,W}^{RH,*}\|_F
\leq \sqrt{\frac{2TW}{\mu}B_{OGD}}.
\end{align}
% \par In addition, the truncation error in \eqref{dis_ineq_star} satisfies $\|M_{t,W}^{*(1)} - M_{t,RH}^*\|_F \leq \varepsilon_W$.
% \begin{align}\label{step_2_1}
%     \|M_{t,W}^{*(1)} - M_{t,RH}^*\|_F \le C \xi^W.
% \end{align}
% {\color{blue}
% \begin{align}\label{step_2_1}
%     \|M_{t,W}^{*(1)} - M_{t,RH}^*\|_F \leq \varepsilon_W.
% \end{align}
% }
\par Substituting \eqref{step_1}, \eqref{step_2_2} and \eqref{step_3} back into the inequality \eqref{dis_ineq_star}, we obtain
\begin{align*}
    &\sum_{t=1}^T \|M_t(W) - M_{t,RH}^*\|_F \\
    \le&  \rho^W \left( \sqrt{W} \sqrt{\frac{2T}{\mu} B_{OGD}} + T \sqrt{W} D_{M} \right) + \mathcal{E}_{T, W}.
\end{align*}
Then, absorbing the constants into the big-O notation gives
\begin{align*}
    &g_T^{RH}(\mathbf{M}) - g_T^{RH}(\mathbf{M}_{RH}^*) \\
    \le& C_g\mathcal{O} \left( \rho^W \sqrt{W T B_{OGD}}+ \rho^W T\sqrt{W} + \mathcal{E}_{T, W} \right),
\end{align*}
which completes the proof.
\end{proof}
}
\par  Based on Lemma \ref{gRH_lem}, we are now ready to establish the main result of this paper, namely, a dynamic policy regret bound for the proposed POC algorithm. 
\begin{theorem}\label{main_thm}
	Under Assumptions \ref{dynamic_ass}, \ref{obj_ass}, and \ref{DAC_ass}, suppose the stepsizes are set as $\delta=\frac{1}{L}$ and $0<\eta\leq 1/L$. Then, the proposed POC algorithm achieves a dynamic policy regret bound of
	\begin{align}\label{reg_boud_1}
	&\mathbb{E}\left[\sum_{t=1}^T c_t(x_t,u_t)-\sum_{t=1}^T c_t(x_t^{\pi_t},u_t^{\pi_t})\right]\notag\\
	&\leq\mathcal{O}\!\left(T(1-\gamma)^{H+1}
	+C_g\rho^W\sqrt{WTB_{OGD}}\right.\notag\\
	&\left.\qquad\quad+C_g\rho^W T\sqrt{W}+C_g\mathcal{E}_{T, W}+\lambda P_T\right),
	\end{align}
	where $\rho=\sqrt{1-\mu\eta}$,  $P_T=\sum_{t=1}^T \|{M}_t^*-{M}_{t-1}^*\|_F$, and
	\[
	B_{OGD} = \frac{D_M^2}{2\delta} + \delta T G_f \left( \frac{G_f}{2} + \lambda \right) + \frac{D_M}{\delta} P_T^{RH}.
	\]
\end{theorem}
\begin{proof}
	We decompose the dynamic policy regret into three terms:
	\begin{align*}
	&\mathbb{E}\left[\sum_{t=1}^T c_t(x_t,u_t)-\sum_{t=1}^T c_t(x_t^{\pi_t},u_t^{\pi_t})\right]\\
	=& \mathbb{E}\left[\sum_{t=1}^T c_t(x_t^K(M_{0:t-1}),u_t^K(M_{0:t}))\right]\\
	&-\mathbb{E}\left[\sum_{t=1}^T c_t(x_t^K(M_{0:t-1}^*),u_t^K(M_{0:t}^*))\right]\\
	=& \underbrace{\mathbb{E}\left[\sum_{t=1}^T c_t(x_t^K(M_{0:t-1}),u_t^K(M_{0:t}))-\sum_{t=1}^T f_t(M_{t-1-H:t})\right]}_{A_T}\\
	&+\underbrace{\mathbb{E}\left[\sum_{t=1}^T f_t(M_{t-1-H:t})-\sum_{t=1}^T f_t(M_{t-1-H:t}^*)\right]}_{B_T}\\
	&+\underbrace{\mathbb{E}\left[\sum_{t=1}^T f_t(M_{t-1-H:t}^*)-\sum_{t=1}^T c_t(x_t^K(M_{0:t-1}^*),u_t^K(M_{0:t}^*))\right]}_{C_T}.
	\end{align*}
	Note that both $A_T$ and $C_T$ correspond to the approximation error introduced by truncating the infinite-memory control problem to a finite memory length $H$. Using Lemma \ref{trunc_approx_err}, we obtain 
	\begin{align}\label{A_C_bound}
	A_T+C_T\leq 4T G D^2 \kappa^3(1-\gamma)^{H+1}.
	\end{align}
	It remains to bound $B_T$, which represents the regret associated with the truncated loss sequence $\{f_t\}_{t=1}^T$. 
    By the Lipschitz continuity of $f_t$ in Lemma \ref{f_t_lips}, we have
    \begin{align*}
        f_t(M_{t-1-H:t})-\tilde{f}_t(M_t)\leq &L_f \sum_{j=1}^{H+1}\|M_{t-j}-M_t\|_F\\
        \leq &L_f\sum_{j=1}^{H+1}\sum_{r=t-j+1}^{t}\|M_r-M_{r-1}\|_F.
    \end{align*}
    Summing over $t$, each switching term is counted at most $(H+2)^2$ times, yielding 
    \begin{align*}
        \sum_{t=1}^T [f_t(M_{t-1-H:t})\!-\!\tilde{f}_t(M_t)]\leq (H\!+\!2)^2 L_f \sum_{t=1}^T \|M_t\!-\!M_{t-1}\|_F.
    \end{align*}
    Using the definition of the regularized objective, we obtain 
	\begin{align}
	B_T=&\sum_{t=1}^T f_t(M_{t-1-H:t})-\sum_{t=1}^T f_t(M_{t-1-H:t}^*)\notag\\
	\leq&\sum_{t=1}^T \tilde{f}_t(M_t)-\sum_{t=1}^T \tilde{f}_t(M_t^*)+\lambda \sum_{t=1}^T \|M_{t-1}-M_t\|_F\notag\\
	&+\lambda \underbrace{\sum_{t=1}^T \|M_{t-1}^*-M_t^*\|_F}_{P_T} \label{decom_RH} \\
	=&g_T^{RH}(\mathbf{M})-g_T^{RH}(\mathbf{M}^*)\notag\\
	&+\lambda\sum_{t=1}^T \|M_{t-1}^*-M_t^*\|_F+\lambda P_T \notag\\
	=&g_T^{RH}(\mathbf{M})-g_T^{RH}(\mathbf{M}_{RH}^*)+g_T^{RH}(\mathbf{M}_{RH}^*)\notag\\
	&-g_T^{RH}(\mathbf{M}^*)+ 2\lambda P_T \notag\\
	\leq &g_T^{RH}(\mathbf{M})-g_T^{RH}(\mathbf{M}_{RH}^*)+2\lambda P_T, \label{B_T_up}
	\end{align}
	where $\lambda=(H+2)^2L_f$, and the last inequality follows from the optimality of $\mathbf{M}_{RH}^*$ as a minimizer of $g_T^{RH}$.
	\par Combining the results in  Lemma \ref{gRH_lem}, \eqref{A_C_bound}, and \eqref{B_T_up}, we have
	\begin{align}\label{c_dif_detres}
	&\mathbb{E}\left[\sum_{t=1}^T c_t(x_t,u_t)-\sum_{t=1}^T c_t(x_t^{\pi_t},u_t^{\pi_t})\right] \notag\\
	\leq& 4T G D^2 \kappa^3(1-\gamma)^{H+1}\notag\\
	&+\mathcal{O}\!\left(C_g\rho^W\sqrt{WTB_{OGD}}
	+C_g\rho^W T\sqrt{W}+C_g\mathcal{E}_{T, W}\right)\notag\\
	&+2\lambda P_T\notag\\
	=&\mathcal{O}\!\left(T(1-\gamma)^{H+1}
	+C_g\rho^W\sqrt{WTB_{OGD}}\right.\notag\\
	&\left.\qquad+C_g\rho^W T\sqrt{W}
	+C_g\mathcal{E}_{T, W}+\lambda P_T\right),
	%=&{\color{blue}\mathcal{O}((1-\mu\eta)^W\sqrt{T\tilde{P}_T}+\tilde{P}_T)}
	\end{align}
	which completes the proof.
	%where $\tilde{P}_T=\sum_{t=1}^T \|\tilde{M}_t^*-\tilde{M}_{t-1}^*\|$. %The last equality follows from Lemma \ref{P_T_lem}, which establishes the relationship between $P_T$ and $\tilde{P}_T$.
\end{proof}
\begin{remark}\label{g_rh_rem}
	{The regularized objective $g_T^{RH}$ is not introduced heuristically. Instead, it is analytically induced by the regret decomposition in \eqref{decom_RH}. Specifically, the Lipschitz continuity argument gives rise to a first-order variation term  $\|M_t-M_{t-1}\|_F$, rather than a squared counterpart. 
    Consequently, the resulting optimization objective contains a non-smooth total-variation regularizer. %Replacing this term with a smooth quadratic regularizer would no longer be consistent with the regret analysis under the considered setting. 
    In addition, the total-variation regularizer suppresses excessive policy variation and aligns naturally with the path-length measure appearing in the dynamic regret analysis.} 
\end{remark}
{
\begin{corollary}\label{reg_corol}
Under the conditions of Theorem \ref{main_thm}, Choose the memory and prediction horizons as
\[
H=\left\lceil c_H\log T\right\rceil,
\qquad
W=\left\lceil c_W\log T\right\rceil,
\]
where
\[
c_H\ge
\frac{1}{|\log(1-\gamma)|},
\qquad
c_W\ge
\frac{1}{|\log\rho|}.
\]
Then the following results hold:
\begin{itemize}
    \item The dynamic policy regret satisfies
    \[
    \mathbb{E}\left[\sum_{t=1}^T c_t(x_t, u_t) \!-\! \sum_{t=1}^T c_t(x_t^{\pi_t}, u_t^{\pi_t})\right]
    \!\le\!
    \widetilde{\mathcal O}(
    1\!+\!P_T\!+\mathcal{E}_{T, W}).
    \]

    \item {If the first-block finite-window error satisfies the uniform \textit{Exponential Decay of Sensitivity} (EDS) bound
    \[
    \|M_{t,W}^{*(1)} - M_{t,RH}^*\|_F \le C_{\mathrm{EDS}}\xi^W,
    \]
    for constants $C_{\mathrm{EDS}}>0$ and
    $\xi\in(0,1)$, and if $c_W$ further satisfies
    \[
    c_W\ge\frac{1}{|\log\xi|},
    \]
    then
    \[
    \mathbb{E}\left[\sum_{t=1}^T c_t(x_t, u_t) - \sum_{t=1}^T c_t(x_t^{\pi_t}, u_t^{\pi_t})\right]
    \le
    \widetilde{\mathcal O}(1+P_T),
    \]}
\end{itemize}
% $$\mathbb{E}\left[\sum_{t=1}^T c_t(x_t, u_t) - \sum_{t=1}^T c_t(x_t^{\pi_t}, u_t^{\pi_t})\right] \le \tilde{\mathcal{O}}(1+P_T),$$
where $\tilde{\mathcal{O}}(\cdot)$ hides logarithmic factors of $T$, and $P_T = \sum_{t=1}^T \|M_{t}^*- M_{t-1}^*\|_F$ is the path-length of the dynamic comparator sequence.
\end{corollary}
\begin{proof}
According to Theorem \ref{main_thm}, the dynamic policy regret is bounded by four main components: the memory truncation error, the optimization error, the local-to-global mismatch error, and the comparator switching cost. We analyze them term by term.

First, under the choice $H=\lceil c_H\log T\rceil$, the $H$-dependent quantities $D$, $G_f$, $L_f$, and $\lambda=(H+2)^2L_f$ grow at most polynomially in $H$. Consequently,
$$C_g=G_f+2\lambda=\operatorname{polylog}(T),$$
and these factors can be absorbed into the $\widetilde{\mathcal O}(\cdot)$ notation. By substituting $H = c_H \log T$, the memory truncation error introduced in the DAC parameterization becomes
\begin{align*}
    \tilde{\mathcal{O}}\left( T(1-\gamma)^{H+1} \right) = &\tilde{\mathcal{O}}\left( T \cdot T^{-c_H |\log(1-\gamma)|} \right) \\
    = &\tilde{\mathcal{O}}\left( T^{1 - c_H |\log(1-\gamma)|} \right).
\end{align*}
By choosing $c_H \ge \frac{1}{|\log(1-\gamma)|}$, the exponent becomes non-positive and this term is $\tilde{\mathcal{O}}(1)$.

\par Then, we bound the optimization error $\tilde{\mathcal{O}}(\rho^{W}\sqrt{WT B_{OGD}})$. Notice that $B_{OGD}$ depends on $P_{T}^{RH}$, which is the path-length of the regularized optimal sequence $M_{RH}^{*}$. Because the decision set $\mathcal{M}$ has a bounded diameter $D_{M}$, the total variation of any sequence within the set is trivially bounded by its absolute worst-case limit, i.e., $P_{T}^{RH} \le T \cdot D_{M} = \mathcal{O}(T)$.
Since the regularization weight is $\lambda = \text{polylog}(T)$ by using $H = \mathcal{O}(\log T)$, substituting this into the initialization bound yields $B_{OGD} = \tilde{\mathcal{O}}(T)$.
Consequently, the overall optimization error simplifies to $ \tilde{\mathcal{O}}\left(\rho^W T \sqrt{W} \right)$.
By choosing the prediction window $W = c_W \log T$, we have $\rho^W = T^{-c_W |\log \rho|}$, and the error term becomes $\tilde{\mathcal{O}}(T^{1 - c_W |\log \rho|})$. By ensuring the constant $c_W \ge \frac{1}{|\log \rho|}$, the exponential decay factor dominates the polynomial growth, making the optimization error $\tilde{\mathcal{O}}(1)$.

% {\color{red}############################
% Third, we bound the optimization error $\mathcal{O}\left( \rho^W \sqrt{T \cdot B_{OGD}} \right)$. From the revised Proposition 4, we have $B_{OGD} = \mathcal{O}(T + P_T)$. Applying the elementary inequality $\sqrt{a+b} \le \sqrt{a} + \sqrt{b}$, the optimization error is bounded by:

% $$\mathcal{O}\left( \rho^W T + \rho^W \sqrt{T P_T} \right).$$

% For the first sub-term, $\rho^W T = T^{1 - c_W |\log \rho|}$. Choosing $c_W \ge \frac{1}{|\log \rho|}$ guarantees that this term is $\mathcal{O}(1)$. For the second sub-term, we apply the AM-GM inequality:

% $$\rho^W \sqrt{T P_T} \le \frac{1}{2} \left( \rho^{2W} T + P_T \right) = \frac{1}{2} \left( T^{1 - 2c_W |\log \rho|} + P_T \right).$$

% Since $c_W \ge \frac{1}{|\log \rho|}$, the first part is $o(1)$, making the entire term bounded by $\mathcal{O}(1 + P_T)$.}

Finally, the penalty parameter for the switching cost is $\lambda = \text{polylog}(T)$. Then, the switching cost of the comparator sequence becomes

$$\mathcal{O}(\lambda P_T) = \tilde{\mathcal{O}}(P_T).$$

Combining these bounded components, the dynamic policy regret is bounded by $\tilde{\mathcal{O}}(1+P_T+\mathcal{E}_{T, W})$.
% $$\mathcal{O}(1) + \mathcal{O}(1) + \tilde{\mathcal{O}}(1 + P_T) + \tilde{\mathcal{O}}(P_T) = \tilde{\mathcal{O}}(1 + P_T).$$
\par If the additional uniform EDS condition holds, then
\[
\mathcal{E}_{T, W}\le C_{\mathrm{EDS}}T\xi^W \le C_{\mathrm{EDS}} T^{1-c_W|\log\xi|}.
\]
Hence, $c_W\ge1/|\log\xi|$ implies
$\mathcal{E}_{T, W}=\mathcal O(1)$,
which gives
\[
\mathbb{E}\left[\sum_{t=1}^T c_t(x_t, u_t) - \sum_{t=1}^T c_t(x_t^{\pi_t}, u_t^{\pi_t})\right]\le \widetilde{\mathcal O}(1+P_T).
\]
This establishes that under a logarithmic prediction window and memory length, the POC algorithm asymptotically tracks the non-stationary comparator sequence up to its path-length.
\end{proof}
}
% \begin{cor}
% 	Under the conditions of Theorem \ref{main_thm}, choose $H=c_H\log T$ with $c_H>\frac{2}{-\log(1-\gamma)}$ and choose $W=c_W\log T$ with $c_W$ sufficiently large such that the exponentially decaying terms in \eqref{reg_boud_1} are lower order. Then the dynamic policy regret of POC satisfies
% 	\[
% 	\sum_{t=1}^T c_t(x_t,u_t)-\sum_{t=1}^T c_t(x_t^{\pi_t},u_t^{\pi_t})
% 	\leq\widetilde{\mathcal{O}}(1+P_T).
% 	\]
% \end{cor}
% \begin{proof}
% 	By Lemma \ref{P_T_lem}, the choice $H=c_H\log T$ gives $\tilde{P}_T=P_T+\mathcal{O}(1)$ and makes the truncation term $T(1-\gamma)^{H+1}$ lower order. Since $\rho\in(0,1)$ and $\xi\in(0,1)$, taking $W=c_W\log T$ with $c_W$ sufficiently large ensures that $\rho^W T\sqrt{W}$ and $T\xi^W$ are lower-order terms. The remaining term $\rho^W\sqrt{WTB_{OGD}}$ is also lower order under the same logarithmic-horizon choice, using the definition of $B_{OGD}$ and the boundedness of the decision set. Finally, $\lambda=(H+2)^2L_f=\mathcal{O}(\log^2T)$, so the term $\lambda P_T$ contributes only logarithmic factors. Therefore, \eqref{reg_boud_1} yields $\widetilde{\mathcal{O}}(1+P_T)$.
% \end{proof}
% \begin{remark}
% The first result in Corollary~\ref{reg_corol} does not require an exponential decay-of-sensitivity property and explicitly retains the approximation error $\varepsilon_W$ induced by the finite prediction window.
% The second result is a conditional specialization, which applies when a uniform EDS bound holds for the TV-regularized window problems.
% \end{remark}
\begin{remark}\label{final_bound_rem}
The derived $\tilde{\mathcal{O}}(1+P_T)$ bound in Corollary \ref{reg_corol} achieves a linear dependence on the comparator path length $P_T$ under the EDS condition. The additional logarithmic factors, arising from the switching regularizer and finite-horizon refinement terms, are inherent to our non-smooth regularization and windowed predictive approach. By selecting logarithmic horizons $H, W = \mathcal{O}(\log T)$, the error terms associated with memory truncation, optimization residuals, and sensitivity decay remain uniformly bounded by $\tilde{\mathcal{O}}(1)$. Consequently, the proposed algorithm achieves the established dynamic policy regret bound using only logarithmic prediction and memory horizons under the EDS condition.
	% The result above shows that the proposed POC algorithm attains a dynamic policy regret bound that is near-optimal up to logarithmic factors. In particular, the ideal benchmark rate for strongly convex dynamic regret is of order $\mathcal{O}(1+P_T)$. Our bound matches this dependence after choosing logarithmic memory and prediction horizons, namely $H=\mathcal{O}(\log T)$ and $W=\mathcal{O}(\log T)$, while the additional logarithmic factors arise from the switching regularization parameter and the finite-horizon refinement terms. The terms $T(1-\gamma)^{H+1}$, $\rho^W T\sqrt{W}\log T$, and $T\xi^W$ quantify, respectively, the truncation error of the DAC reduction, the overall residual optimization error of the proximal-gradient refinement, and the finite-horizon sensitivity error. These terms become lower order under the logarithmic choices of $H$ and $W$. Therefore, the theorem can be interpreted as establishing a logarithmically near-optimal dynamic policy regret guarantee under the EDS condition, rather than an exact $\mathcal{O}(P_T)$ bound without additional structural assumptions.
\end{remark}
{
\begin{remark}
   The proposed POC algorithm reduces the per-step optimization burden relative to solving a full receding-horizon problem to convergence. 
   Standard MPC typically requires solving a multi-stage optimization problem at each time step, often involving multiple interior-point or Newton-type iterations. For dense and unstructured formulations with fixed per-stage dimensions, the solution of the associated KKT system can scale as $\mathcal{O}(W^3)$ with the prediction horizon $W$, while the total computation also depends on the number of iterations and the prescribed accuracy.
   In contrast, POC method does not solve the full multi-stage objective to convergence at each time step, but instead performs $W$ joint proximal-gradient refinements. Each refinement exploits the separable structure of the smooth surrogate losses together with the chain structure of the total-variation regularizer and the block-wise structure of the DAC constraints. When the prediction horizon is chosen as $W=\mathcal{O}(\log T)$ in Corollary \ref{reg_corol}, POC requires only $\mathcal{O}(\log T)$ outer refinement iterations per step. This provides a computationally structured alternative to repeatedly solving a full receding-horizon problem to high accuracy. %Our algorithm requires merely $\mathcal{O}(\log T)$ refinement updates per time step, offering a reduction in computational burden.
\end{remark}

}

\section{Simulation Results}\label{sim_sec}
This section evaluates the proposed POC algorithm on a motorized LiDAR
sensing (MLiS) system using sequences from the MCD dataset. The MLiS system consists of a LiDAR sensor mounted on a motorized platform, where the motor controls the rotation of the LiDAR to capture panoramic depth information. The control objective is to optimize the motor's rotation speed to enhance the quality of LiDAR sensing while maintaining accurate localization and efficient scanning.  %captures the essential characteristics of real-world MLiS systems while allowing for controlled experimentation and analysis. The MLiS system consists of a LiDAR sensor mounted on a motorized platform, where the motor controls the rotation of the LiDAR to capture panoramic depth information. The control objective is to optimize the motor's rotation speed to enhance the quality of LiDAR sensing while maintaining accurate localization and efficient scanning.
\par Most existing MLiS systems employ fixed motor control strategies, which are often unable to adapt well to dynamic environments and result in suboptimal sensing performance. The proposed POC algorithm dynamically adjusts the motor speed using both sensing-quality predictions and closed-loop motor feedback. Specifically, the LiDAR odometry and local mapping modules are used to construct the environment-dependent observability measure and its short-term predictions. These predictions define the lookahead cost functions used by POC to update the DAC parameter $M_t$. The tracking-error state and the recovered disturbance history are then mapped into the motor-speed command through the DAC control law. The resulting closed-loop workflow is illustrated in Fig. \ref{workflow_fig}.
% The proposed POC algorithm is designed to dynamically adjust motor control based on real-time observed LiDAR data and odometry estimates to enhance both localization accuracy and scanning efficiency. Specifically, the LiDAR data stream from the motorized LiDAR system is processed by a LiDAR odometry module, whose estimates are then used to compute the control input for the motor \cite{Li2026AEOS}. This input subsequently adjusts the LiDAR’s viewing direction to optimize sensing quality. The overall workflow of the MLiS control model is illustrated in Fig. \ref{workflow_fig}.
\par {Let $\theta_t$ denote the unwrapped measured rotor angle and $\omega_{\mathrm{pre}}$ denote the nominal angular velocity. The corresponding reference angle evolves according to 
\[
\theta_{t+1}^{\mathrm{ref}}
=
\theta_t^{\mathrm{ref}}
+\omega_{\mathrm{pre}}\Delta t.
\]
Here, $\omega_t$ denotes the motor-speed command applied during the sampling interval $[t,t+1)$. We introduce the tracking-error state and incremental control input as
\[
x_t:=\theta_t-\theta_t^{\mathrm{ref}},
\qquad
u_t:=\omega_t-\omega_{\mathrm{pre}}.
\]
The simulated motor dynamics are
\[
\theta_{t+1}
=
\theta_t+\omega_t\Delta t+w_t,
\]
where $w_t$ represents the aggregate motor-tracking error, encoder uncertainty, and unmodeled discretization effects. Consequently, the tracking-error dynamics become
\[
x_{t+1}=x_t+\Delta t\,u_t+w_t.
\]
Therefore, the motorized LiDAR model is a scalar instance of the linear system \eqref{linear_dyn_sys} with $A=1$ and $B=\Delta t$.
\par At the beginning of time step $t$, the most recent disturbance is recovered from the measured state transition as
\begin{align}
    w_{t-1}
    &=
    x_t-Ax_{t-1}-Bu_{t-1}                                      \notag\\
    &=
    \theta_t-\theta_{t-1}
    -\Delta t\,\omega_{t-1} .
    \label{eq:lidar_disturbance_recovery}
\end{align}
Therefore, only previously observed quantities are used in the control law. We set $w_s=0$ for $s<0$.
\par The motor command is generated by the same DAC parameterization:
\begin{equation}
    u_t
    =
    -Kx_t
    +
    \sum_{i=1}^{H}
    M_t^{[i-1]} w_{t-i}.
    \label{eq:lidar_dac}
\end{equation}
Recalling that $u_t=\omega_t-\omega_{\mathrm{pre}}$, the
physical motor-speed command is therefore
\begin{equation}
\omega_t
=
\omega_{\mathrm{pre}}
-Kx_t
+\sum_{i=1}^{H}M_t^{[i-1]} w_{t-i}.
\label{eq:motor_speed_command}
\end{equation}
Thus, POC optimizes the DAC parameter $M_t$, which is mapped into the motor-speed command through \eqref{eq:motor_speed_command}, while $K$ remains fixed.
% Here, $K$ is a fixed stabilizing feedback gain, while
% \[
%     M_t
%     =
%     \bigl(
%     M_t^{[0]},\ldots,M_t^{[H-1]}
%     \bigr)
% \]
% is the time-varying DAC parameter optimized by POC. 
Since $A-BK=1-\Delta t K$, the scalar closed-loop stability condition
is
\begin{equation}
    |1-\Delta t K|<1.
    \label{eq:scalar_stability}
\end{equation}
}
% The system state at time $t$ is represented by the viewing direction $\theta_t$, which evolves according to the motor angular velocity control input $v_t$ as follows:
% \begin{align}
% \theta_{t+1}=\theta_t+v_t\Delta_t,
% \end{align}
% where $v_t$ is the motor angular velocity (rad/s) to be controlled. 
{Following the work \cite{li2025ua_mpc}, the online cost function at each time step is formulated as a weighted sum of localization uncertainty and control effort. The sensing quality depends on the viewing direction $\theta$ via an environment-dependent observability function $U(\theta)$.
The physical sensing-control cost is
\begin{align}\label{f_t_sim_rev}
\ell_t(\theta,\omega)
=
\alpha\|U_t(\theta)\|^2
+
\beta\|\omega-\omega_{\mathrm{pre}}\|^2,
\end{align}
where $\alpha, \beta>0$ are weighting parameters.
Under the tracking-error coordinates, the corresponding state--input cost is
\[
c_t(x,u)
:=
\ell_t(\theta_t^{\mathrm{ref}}+x,
       \omega_{\mathrm{pre}}+u)
=
\alpha
\|U_t(\theta_t^{\mathrm{ref}}+x)\|^2
+
\beta\|u\|^2.
\]
}
% \par Following the work \cite{li2025ua_mpc}, the online cost function at each time step is formulated as a weighted sum of odometry accuracy and control effort. We denote the motor angular velocity by $\omega_t$ in the simulation and write the stage cost as
% \begin{align}\label{f_t_sim}
% &f_t(\omega_t)=\alpha \|U(\theta_t)\|^2+\beta \|\omega_t-\omega_{pre}\|^2,
% \end{align}
% where $U(\theta)$ denotes the observability uncertainty, $v_{pre}$ is the preset reference angular velocity of the motor, and $\alpha, \beta>0$ are weighting parameters.
\par To make the online optimization tractable, we approximate the observability function by a piecewise-linear surrogate $U'(\theta)$:
\begin{align}
U'(\theta)=&\left(1-\frac{\theta}{\Delta \theta}+\left\lfloor \frac{\theta}{\Delta \theta} \right\rfloor \right) U_{\left\lfloor \frac{\theta}{\Delta \theta}\right\rfloor}\notag\\
&+\left(\frac{\theta}{\Delta \theta}-\left\lfloor \frac{\theta}{\Delta \theta}\right\rfloor\right) U_{\left\lfloor \frac{\theta}{\Delta \theta}\right\rfloor+1}, \notag
\end{align}
as illustrated in Fig. \ref{surr_func_fig}. The first term in \eqref{f_t_sim_rev} measures localization uncertainty induced by the current viewing direction, whereas the second term penalizes deviation from the nominal motor speed. {The theoretical analysis assumes globally strongly convex and smooth stage costs, whereas the practical observability objective is locally approximated from scene-dependent data. The simulation serves as an empirical evaluation of the proposed mechanism under a general sensing objective.}
\begin{figure}
	\centering
	\includegraphics[width=8cm]{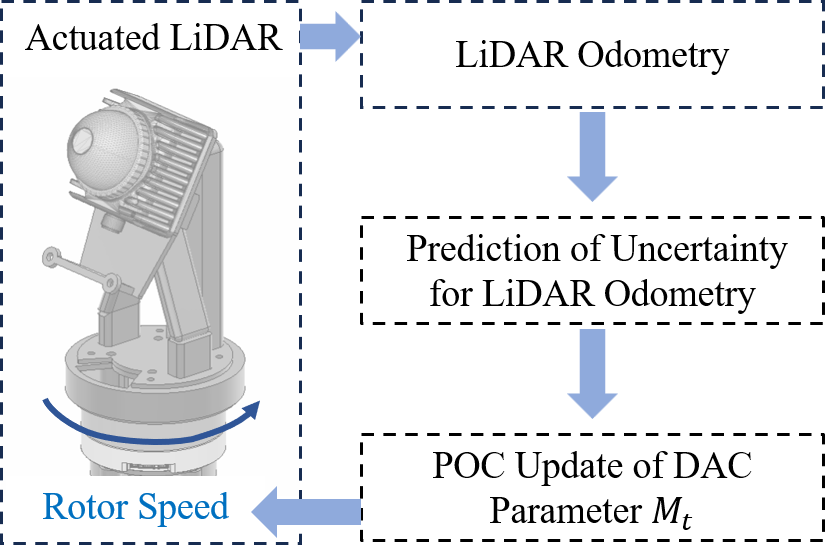}
	\caption{Closed-loop implementation of POC in the motorized LiDAR system.}
	\label{workflow_fig}
\end{figure}
\begin{figure}
	\centering
	\includegraphics[width=8cm]{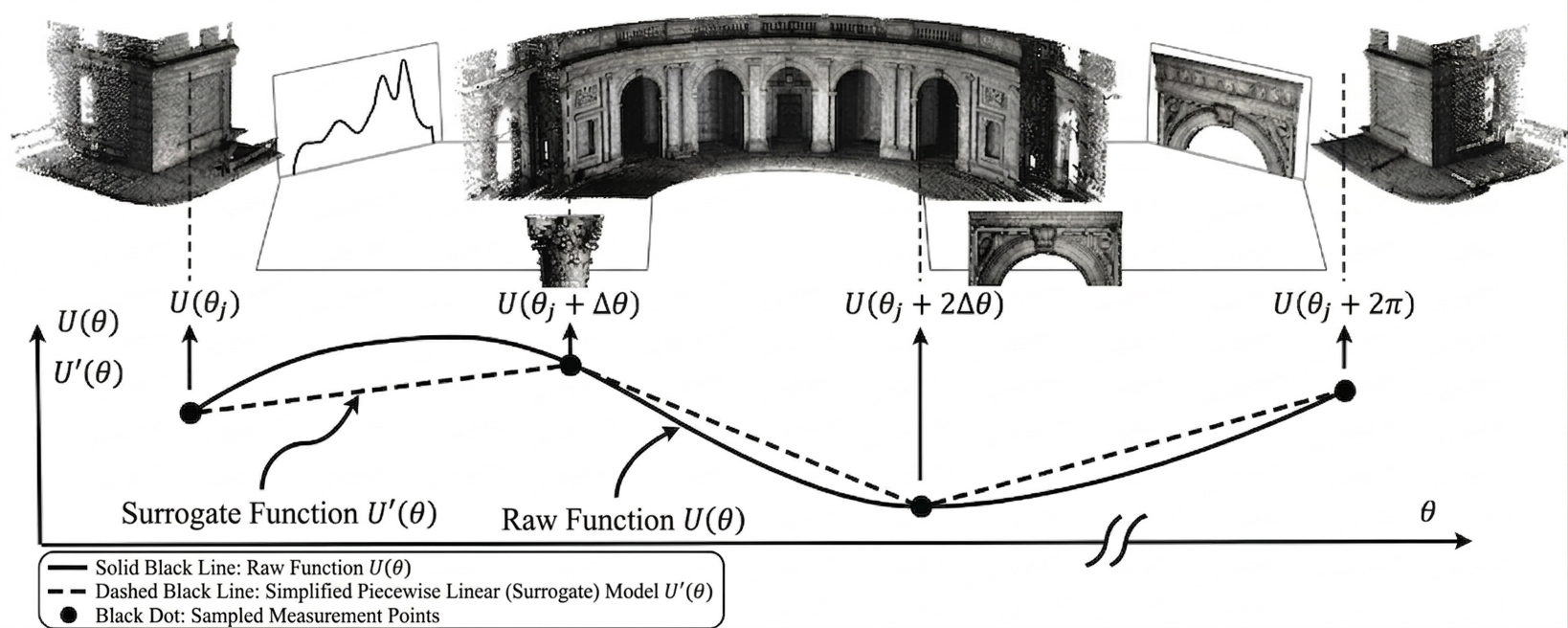}
	\caption{The piecewise linearized surrogate function $U'$.}
	\label{surr_func_fig}
\end{figure}
\par This stage cost is explicitly time-varying because the sensing quality depends on the local scene geometry encountered along the platform trajectory. Since the underlying map may be non-convex, we use a local piecewise-linear approximation $U'(\theta)$ around the current direction $\theta_t$ to construct a tractable surrogate for the POC updates.
At time $t$, the uncertainty-prediction module provides the predicted observability profiles over the next $W$ steps. These profiles define the predicted lookahead state-input costs $\{c_s\}_{s=t}^{t+W-1}$ and, through the truncated DAC model, the corresponding surrogate losses $\{\tilde f_s\}_{s=t}^{t+W-1}$. The POC algorithm then performs the windowed proximal-gradient updates to obtain $M_t$, which is converted into the motor-speed command through  \eqref{eq:motor_speed_command}.
For the simulation setup, we set the memory length to $H=10$ and the prediction horizon to $W=8$. The cost weights are configured as $\alpha=1000$ and $\beta=1$, with a baseline preset speed of $\omega_{\mathrm{pre}}=1.8$ rad/s. Evaluations are conducted across the NTU, KTH, and TUHH scenes from the MCD dataset.
\par We evaluate the proposed algorithm and baseline methods from two perspectives: (1) localization-odometry (LO) accuracy and (2) scanning efficiency. LO accuracy is quantified using the Absolute Translation Error (ATE):
$$
\mathrm{ATE}=\sqrt{\frac{1}{N} \sum_{i=1}^N\left\|\mathbf{r}_i^{\mathrm{est}}-\mathbf{r}_i^{\mathrm{gt}}\right\|^2},
$$
where $\mathbf{r}_i^{\text {est }}$ and $\mathbf{r}_i^{\mathrm{gt}}$ denote the estimated and ground truth positions, respectively, and $N$ represents the total number of pose estimates.
Scanning efficiency is measured by the average number of voxels (0.5 m each) scanned every 5 seconds, defined as completeness (CMPLT):
$$
\mathrm{CMPLT}=\frac{1}{Q} \sum_{q=1}^Q v_q,
$$
where $v_q$ is the number of scanned voxels during the $q$-th scan period, and $Q$ is the total number of scan periods. This dual-metric framework provides a comprehensive assessment of both accuracy and efficiency.

\par We compare the proposed  POC algorithm with two baseline motor control strategies commonly used in existing systems: constant-speed control (1.8 rad/s) and Uncertainty-Aware Model Predictive Control (UA-MPC) \cite{li2025ua_mpc}. The UA-MPC method formulates the control problem as a finite-horizon optimization that explicitly incorporates a sensing uncertainty metric into the objective and solves a nonlinear program at each step to obtain control actions. Consistent with \cite{li2025ua_mpc}, we set the weights to $\alpha=1000$ and $\beta=1$ to balance both objectives. The MLiS system aims to dynamically adjust the rotation speed in response to unbalanced features in the panoramic depth map, focusing on regions with rich features while maintaining a trade-off between scanning efficiency and LO accuracy. This demonstrates the adaptability and effectiveness of control methods in diverse scenarios.

\begin{table}
	\caption{Evaluation on Simulation Dataset}
	\label{eva_table}
	\centering
	\resizebox{\columnwidth}{!}{%
	\begin{tabular}{|l|l|l|l|l|l|l|}
		\hline \multirow{2}{*}{} & \multicolumn{2}{|c|}{NTU} & \multicolumn{2}{|c|}{KTH} & \multicolumn{2}{|c|}{TUHH} \\
		& ATE & CMPLT & ATE & CMPLT & ATE & CMPLT \\
		\hline Constant & 5.208 & 67574 & $\underline{6.201}$ & 37148 & 8.735 & 41013 \\
		%\hline Zero-Speed (speed $=0$ ) & 4.71 & 51977 & 15.3 & 22529 & 84.64 & 24737 \\
		\hline UA-MPC & 5.093 & 70474 & 6.411 & $\underline{61541}$ & 7.852 & $\underline{97670}$ \\
		\hline POC & $\underline{3.236}$ & \underline{74851} & 6.219 & 42999 & $\underline{6.213}$ & 42344 \\
		\hline
	\end{tabular}}
\end{table}

\begin{figure}[htbp]
	\centering
	\subfigure[NTU]{
		\includegraphics[width=8cm]{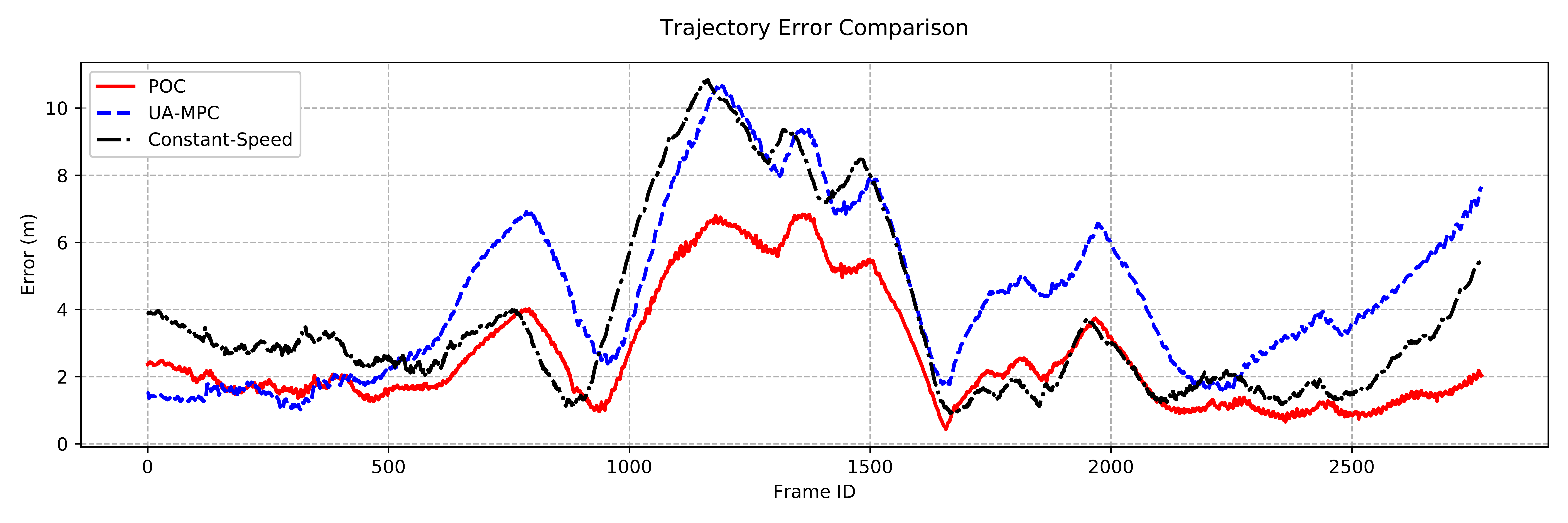}
	}
	\centering
	\subfigure[KTH]{
		\includegraphics[width=8cm]{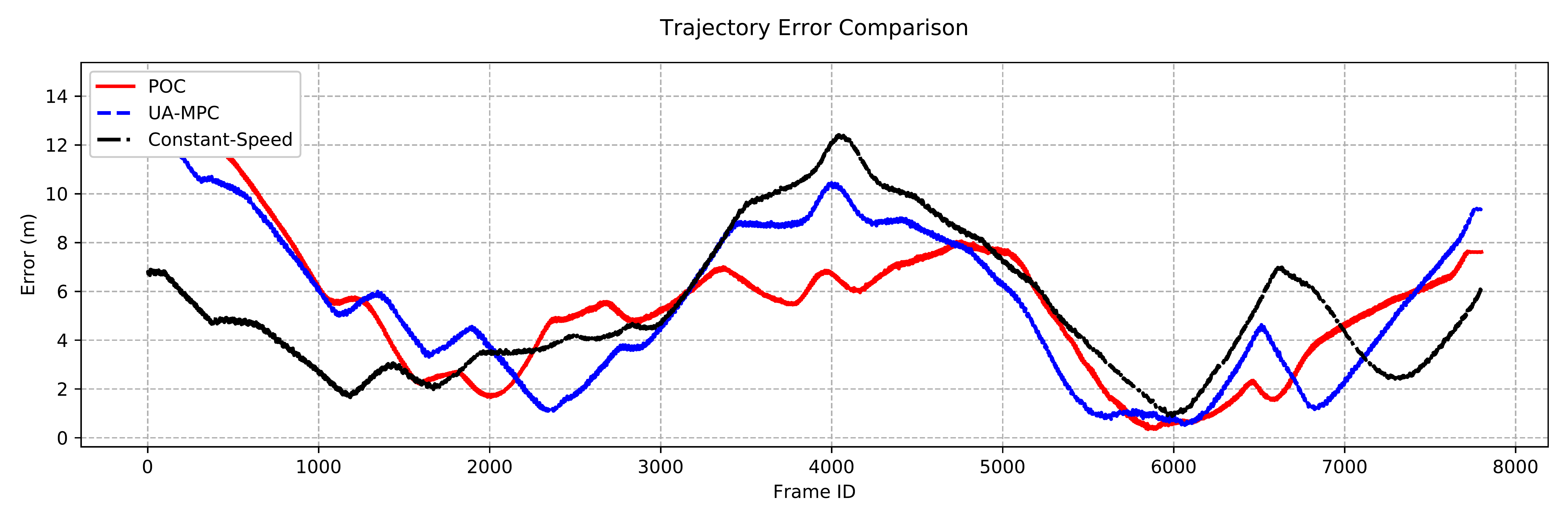}
	}
	\subfigure[TUHH]{
		\includegraphics[width=8cm]{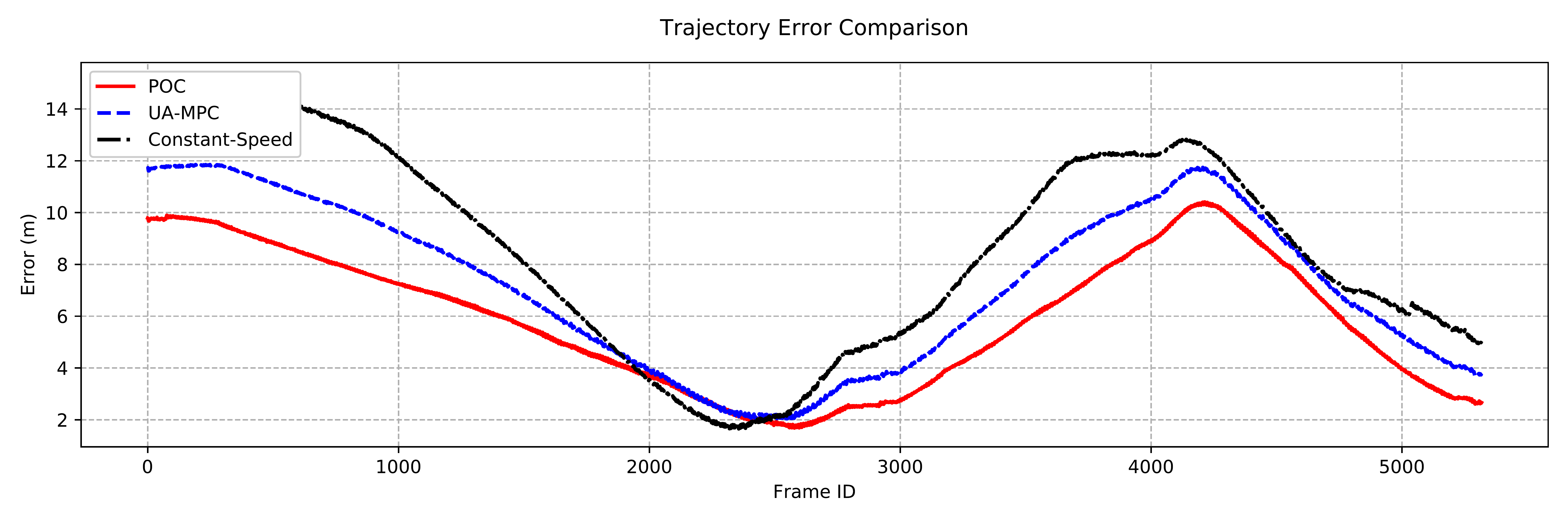}
	}
	\caption{ Trajectory error comparison across different scenes. }
	\label{err_fig}
\end{figure}

\begin{figure*}[htbp]
	\centering
	\includegraphics[width=0.9\textwidth]{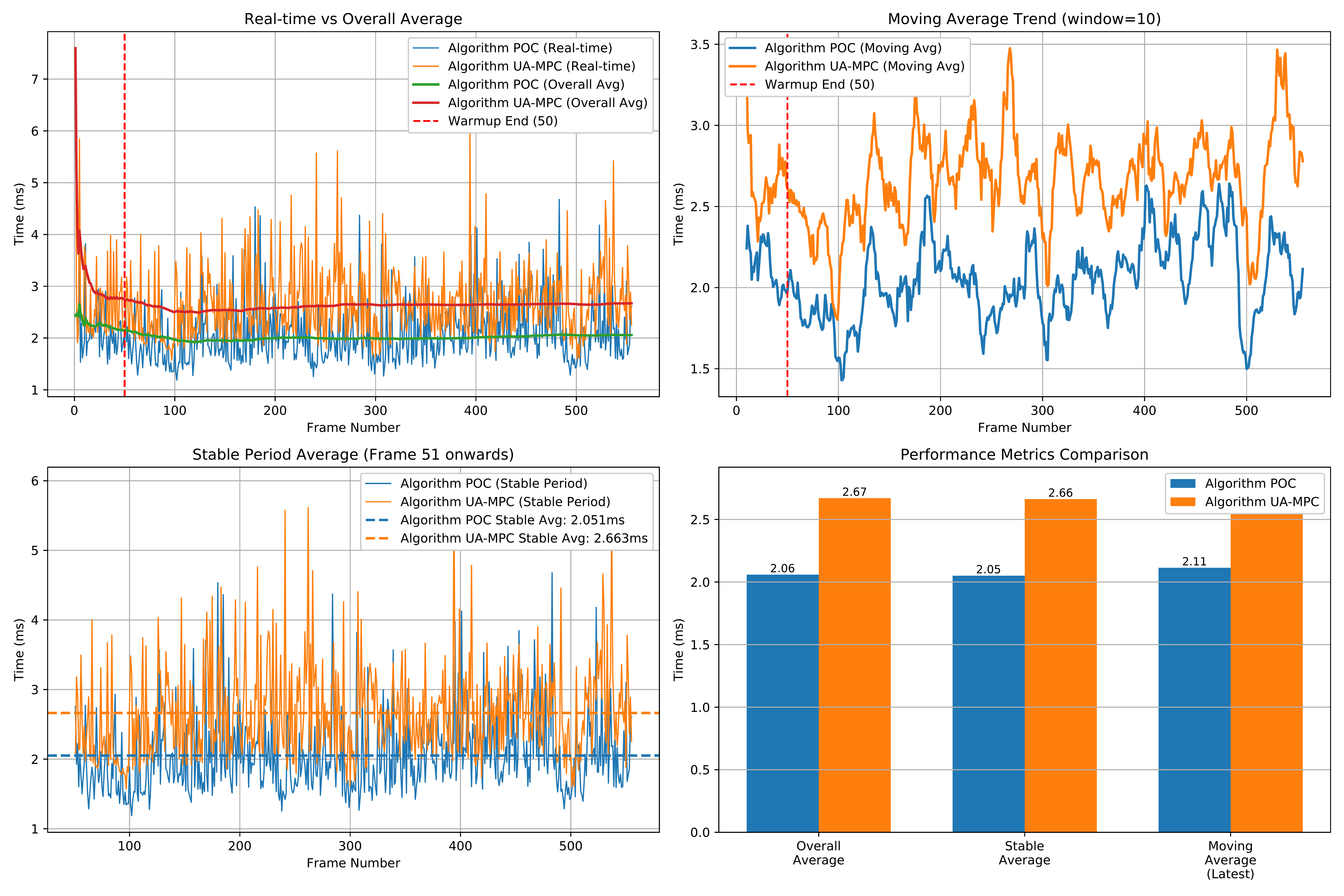}
	\caption{Runtime comparison of POC and UA-MPC on NTU dataset}
	\label{time_fig}
\end{figure*}
\par Table \ref{eva_table} presents the performance comparison among the proposed POC algorithm, UA-MPC \cite{li2025ua_mpc}, and constant-speed control across three scenarios. The results show that POC achieves the lowest ATE in the NTU and TUHH scenes. In KTH, its ATE is close to that of constant-speed control and lower than that of UA-MPC. In terms of CMPLT, POC outperforms constant-speed control in all three scenes and achieves the highest value in NTU, whereas UA-MPC achieves higher CMPLT in KTH and TUHH. These results demonstrate a scene-dependent trade-off between localization accuracy and scanning completeness. %The trajectory error ATE comparison in Fig. \ref{err_fig} further illustrates that POC consistently maintains lower trajectory errors compared to the baseline methods, particularly in NTU and TUHH scenes. 
The trajectory-error comparison in Fig. \ref{err_fig} further shows that POC achieves lower errors than UA-MPC across the three scenes and provides particularly clear improvements in NTU and TUHH, while remaining comparable to constant-speed control in KTH. {
Although POC and UA-MPC use the same system model and predictive information, POC performs incremental DAC updates with total-variation regularization rather than directly optimizing the finite-horizon uncertainty surrogate. This regularization discourages abrupt policy variations and provides a plausible explanation for the lower ATE observed with POC.
}
\par To further demonstrate the practical efficiency of the proposed POC algorithm, we compare its computation time with the UA-MPC on the NTU dataset. The results are shown in Fig. \ref{time_fig}. The steady-state average computation time of POC is significantly lower than that of UA-MPC, indicating that POC is more computationally efficient and better suited for real-time control applications. Specifically, POC achieves a significantly lower steady-state average computation time (2.051 ms vs. 2.663 ms), corresponding to a 23\% improvement, while also exhibiting reduced variability (standard deviation 0.564 ms vs. 0.680 ms). The real-time and moving-average curves further show that, after the warm-up phase, POC maintains a lower and less variable per-step computation time than UA-MPC. %Additionally, its lower worst-case latency indicates improved robustness. 
Overall, these results demonstrate that POC provides a more computationally efficient implementation for real-time motorized LiDAR control.

\section{Conclusion}\label{conclu_sec}
This paper has proposed a predictive online control (POC) algorithm for linear dynamical systems with adversarial and time-varying cost functions. 
By integrating the DAC parameterization with a windowed receding-horizon optimization framework,  the proposed approach effectively incorporates short-term predictions into online control while explicitly regulating policy variation. Under standard assumptions, we established dynamic policy regret guarantees for the proposed algorithm and explicitly characterized the coupled effects of the prediction horizon and the DAC memory truncation on the regret performance. 
Simulation results on a motorized LiDAR sensing system further demonstrated the practical effectiveness of the proposed algorithm in improving localization accuracy while achieving a favorable trade-off between localization performance and scanning completeness. Future work will focus on extending the framework to general convex cost functions and on developing a unified architecture that jointly performs prediction and online control.

\bibliographystyle{IEEEtran}
\bibliography{refer}

\end{document}